\documentclass{amsart}[12 pt]
\usepackage{amsmath, amsfonts, amssymb, amsthm, euscript, amscd, latexsym, mathrsfs, bm, color, bbm, xcolor,mathtools}
 \usepackage{titlesec}
 \usepackage{enumitem}
\usepackage{lipsum}
 \usepackage[pdftex,colorlinks=true,
                       pdfstartview=FitV,
                       linkcolor=blue,
                       citecolor=blue,
                       urlcolor=blue,
           ]{hyperref}

\def\eee#1{ \begin{equation*} \begin{cases} #1 \end{cases} \end{equation*} }

\def\aa#1{ \begin{align*} #1 \end{align*} }
\def\aaa#1{ \begin{align} #1 \end{align} }
\def\mm#1{ \begin{multline*} #1 \end{multline*} }
\def\mmm#1{ \begin{multline} #1 \end{multline} }
\def\mml#1{ \begin{multlined}[b] #1 \end{multlined} }

\newtheorem{thm}{\sc Theorem}
\newtheorem{lem}{\sc Lemma}
\newtheorem{cor}{\sc Corollary}

\newtheorem{pro}{\sc Proposition}
\newtheorem{rem}{\sc Remark}

\newtheorem*{thm*}{\sc Theorem}

\newcommand{\eps}{\varepsilon}
\newcommand{\pl}{\partial}
\newcommand{\gt}{\geqslant}
\newcommand{\lt}{\leqslant}

\renewcommand{\leq}{\leqslant}
\newcommand{\te}{\theta}
\newcommand{\sub}{\subset}
\newcommand{\dl}{\delta}
\newcommand{\al}{\alpha}
\newcommand{\gm}{\gamma}
 \newcommand{\Gm}{\Gamma}
 \newcommand{\Dl}{\Delta}
 \newcommand{\la}{\lambda}
 \newcommand{\La}{\Lambda}
 \newcommand{\sg}{\sigma}
\newcommand{\dd}{\diagdown}

\newcommand{\mc}{\mathcal}
\newcommand{\Om}{\Omega}
\newcommand{\mf}{\mathfrak}

\newcommand{\C}{{\rm C}}

\newcommand{\td}{\tilde}
\newcommand{\ox}{\otimes}

\newcommand{\E}{\mathbb E}
\newcommand{\we}{\wedge}

\newcommand{\x}{\times}
\newcommand{\mto}{\mapsto}
\newcommand{\PP}{\mathbb P}

\newcommand{\tr}{{\rm tr}}

\newcommand{\rf}{\eqref}
\newcommand{\bi}{\begin{itemize}}
\newcommand{\ei}{\end{itemize}}
\newcommand{\biz}{\begin{itemize}[leftmargin=*]}

\newcommand{\ovl}{\overline}
\newcommand{\F}{{\mathbb F}}
\newcommand{\G}{{\mathbb G}}

\newcommand{\lap}{\Delta}

\newcommand{\lb}{\label}
\newcommand{\fdot}{\,\cdot\,}
\newcommand{\hsp}{\hspace{-2mm}}
\newcommand{\hs}{\hspace{-1mm}}
\newcommand{\cp}{\varkappa}

\def\RR{\mathbb{R}} 
\def\begeq{\begin{equation} \begin{cases}} 
\def\endeq{ \end{cases} \end{equation}}
\def\eq#1{ \begeq #1 \endeq }
\def\bege{\begin{equation*} \begin{cases}} 
\def\ende{ \end{cases} \end{equation*}}
\def\eqq#1{ \bege #1 \ende}

\newcommand{\ro}{\varrho}

\newcommand{\sig}{\varsigma}
\newcommand{\tet}{\vartheta}

\def\Rnu{{\mathbb R}}

\def\ffi{\varphi}

\def\suml{\sum\limits}

\def\com#1{}

\long\def\symbolfootnote[#1]#2{\begingroup%
\def\thefootnote{\fnsymbol{footnote}}\footnote[#1]{#2}\endgroup}
\titleformat{\section}[hang]{\large\bfseries}{\thesection.}{1ex}{}{}
\titleformat{\subsection}[hang]{\normalsize\bfseries}{\thesubsection}{2ex}{}{}
\titleformat{\subsubsection}[hang]{\small\bfseries}{\thesubsubsection}{2ex}{}{}

\include{srctex.sty}

\allowdisplaybreaks

\begin{document}


\title[FBSDEs with jumps and classical solutions to nonlocal PDEs]{Forward--backward SDEs with jumps and  
classical solutions to nonlocal quasilinear parabolic PDEs}

\author{Evelina Shamarova}
\address{Departamento de Matem\'atica, Universidade Federal da Para\'iba, Jo\~ao Pessoa, Brazil}
\email{evelina@mat.ufpb.br}

\author{Rui S\'a Pereira}
\address{Departamento de Matem\'atica, Universidade do Porto, Rua do Campo Alegre, Porto, Portugal}
\email{manuelsapereira@gmail.com}

\maketitle

{\footnotesize
Partial integro-differential equations

\vspace{1mm}

}

\begin{abstract}
We obtain an existence and uniqueness theorem for fully coupled 
forward--backward SDEs (FBSDEs) with jumps via the classical solution
to the  associated  quasilinear parabolic  partial integro-differential equation (PIDE),
and provide the explicit form of the FBSDE solution. 
Moreover, we embed the associated PIDE into a suitable class of non-local quasilinear parabolic PDEs
which allows us to extend the methodology of Ladyzhenskaya et al (1968)
to non-local PDEs of this class. Namely, we obtain 
the existence and uniqueness of a classical solution to both the Cauchy problem and 
the  initial--boundary value problem for non-local quasilinear parabolic second-order PDEs. 
\end{abstract}



\section{Introduction}
One of the well-known tools to solve FBSDEs driven by a Brownian motion is their link
to quasilinear parabolic PDEs which, by means of It\^o's formula, allows  to obtain
the explicit form of the FBSDE solution via the classical solution to the associated PDE \cite{CS, ma1, pardeux-tang, peng}. 
However, if we are concerned with FBSDEs with jumps, the associated PDE becomes a PIDE whose
coefficients contain non-local dependencies on the solution.
To the best of our knowledge, there are no results on the solvability, in the classical sense, of PIDEs
appearing in connection to FBSDEs with jumps.  

In this work, we obtain the existence and uniqueness of a classical solution
for a class of non-local quasilinear parabolic PDEs, which includes PIDEs of interest, and apply this
result to obtain the existence and uniqueness of solution to fully coupled FBSDEs driven by an 
$n$-dimensional Brownian motion 
and a compensated Poisson random measure on an arbitrary time interval $[0,T]$:
\eq{
\lb{eqFBSDE}
X_t = x + \int_0^t f(s,X_s,Y_s,Z_s, \td Z_s) \,ds +\int_0^t \sg (s,X_s,Y_s)\, dB_s \\
\hspace{6.4cm} + \int_0^t \int_{{\Rnu^l_*}} \ffi(s,X_{s-},Y_{s-},y)\, \td N(ds, dy),\\
Y_t = h(X_T) + \int_t^T g(s,X_s,Y_s,Z_s, \td Z _s) \,ds -\int_t^T Z_s \, dB_s 
- \int_t^T \int_{{\Rnu^l_*}} \td Z_s (y)  \,  \td N(ds, dy).
}
Above, the forward SDE is $n$-dimensional and the backward SDE (BSDE) is $m$-dimensional;
 the coefficients $f(t,x,u,p,w)$, $g(t,x,v,p,w)$, $\sg(t,x,u)$, and $\ffi(t,x,u,y)$ are functions of appropriate dimensions 
 whose argument $(t,x,u,p,w)$ belongs to the space
 $[0,T] \x \Rnu^n\x \Rnu^m \x \Rnu^{m \x n} \x L_2(\nu, {\Rnu^l_*} \to \Rnu^m)$, where
 $\nu$ is the intensity of the Poisson random measure involved in \rf{eqFBSDE} and
 $\Rnu^l_* = \Rnu^l -\{0\}$.
 
Our second object of interest is 
the following $\RR^m$-valued non-local quasilinear parabolic  PDE  on $[0,T]\x \Rnu^n$  
associated to FBSDE \rf{eqFBSDE}
\aaa{
\label{PDE}
- \sum_{i,j=1}^{n} a_{ij} (t,x,u)\pl^2_{x_i x_j}u + \sum_{i=1}^n a_i (t,x,u, \pl_x u, \tet_u)\pl_{x_i} u 
+ a(t,x,u,  \pl_x u,\tet_u)  +  \pl_t u=0.
}
The coefficients of \rf{PDE}  are expressed via the coefficients of \rf{eqFBSDE} as follows
\eq{
\lb{coeff}
&a_{ij}(t,x,u)=   \frac12 \sum\nolimits_{k=1}^n \sg_{ik}\sg_{jk}(T-t,x,u), \\
&a_i(t,x,u,p,w)=   \int_Z \ffi_i (T-t,x,u,y)   \nu(dy) -   f_i\big(T-t,x,u, p\, \sg(T-t,x, u), w\big), \\
&a(t,x,u,p,w)= -   g\big(T-t,x,u, p\,\sg(T-t,x,u), w\big) - \int_Z   w(y)  \, \nu (dy), \\
&\tet_u(t,x) = u(t,x+\ffi(T-t,x,u(t,x),\fdot)) - u(t,x),
}
where the support  $Z$ of the function $y \mto \ffi(t,x,u,y)$ is assumed to have a finite $\nu$-measure
for each $(t,x,u)\in [0,T]\x \Rnu^n \x \Rnu^m$.
In \rf{PDE}, $\pl^2_{x_i x_j}\!u$, $\pl_{x_i}\!u$, $\pl_t u$, $u$, and $\tet_u$ are evaluated at $(t,x)$.
Non-local PDE \rf{PDE} is assumed to be uniformly  parabolic, i.e.,
for all $\xi\in\Rnu^n$, it holds that
$\hat \mu(|u|) |\xi|^2 \leq \sum_{i,j=1}^n a_{ij}(t,x,u) \xi_i \xi_j \leq \mu(|u|) |\xi|^2$,
where $\mu$ and $\hat \mu$ are non-decreasing, and, respectively, non-increasing functions.

BSDEs and FBSDEs  with jumps have been studied by many authors, e.g., 
\cite{pardoux97,  li17, li-peng09, LiWei, nualart,  tang, wu99, wu03}. Existence and uniqueness results
for fully coupled FBSDEs with jumps were previously obtained  in  \cite{wu99},  \cite{wu03}, 
and, on a short time interval, in  \cite{LiWei}.
The main assumption in \cite{wu99} and  \cite{wu03} is the so-called monotonicity assumption
(see, e.g., \cite{wu99}, p. 436, assumption (H3.2)). This is a rather technical condition
that fails to hold for most of naturally appearing functions. 

We remark that our result on the existence and uniqueness of solution to FBSDE \rf{eqFBSDE} holds
on a time interval of arbitrary length and without any sort of monotonicity assumptions.
Our assumptions on the FBSDE coefficients are formulated in a way that makes it possible 
to solve, in the classical sense,
the associated PIDE, which is a particular case of
non-local PDE \rf{PDE}.


Importantly, we obtain a link between
the solution to FBSDE \rf{eqFBSDE} and the solution to the associated PIDE.
A similar link in the case of FBSDEs driven by a Brownian motion was established by Ma, Protter, and Yong \cite{ma1}, 
and the related result on the solvability of FBSDEs is known as the four step scheme. 
The main tool to establish this link (and, consequently,
 to solve Brownian FBSDEs) was the existence of a classical solution to traditional 
 quasilinear parabolic second-order  PDEs, 
 which, however, is well known;
 and as the main reference used in \cite{ma1}, stands the monograph of Ladyzhenskaya et al \cite{lady}. 
 Since the consideration of FBSDEs with jumps leads to PDEs of type \rf{PDE}, i.e., containing
the non-local dependence $\tet_u$, the theory developed in \cite{lady} is not applicable anymore. 

It is rather doubtless that a L\'evy analog of the four step scheme is an 
interesting problem on its own; yet due to the technical complexity of the 
proofs, a direct generalization of the results of \cite{lady} to non-local PDEs  is not obvious.
Indeed, it has been more than twenty years since the original four step scheme was published.

Thus, this article has the following two main contributions. First of all,
we define a class of non-local quasilinear parabolic PDEs containing the PIDE associated to FBSDE \rf{eqFBSDE}
and  establish the existence and uniqueness 
of a classical solution to the Cauchy problem and the initial--boundary value problem
for PDEs of this class; and, secondly, we prove
 the existence and uniqueness theorem for fully coupled FBSDEs with jumps \rf{eqFBSDE}
and provide the formulas that express the solution to FBSDE \rf{eqFBSDE} via the solution to 
associated non-local PDE \rf{PDE} with the coefficients and the function $\tet_u$ given by \rf{coeff}.
The major difficulty of this work appears in obtaining the first of the aforementioned results.

The following scheme is used to obtain the existence and uniqueness result for non-local PDEs.
We start with the initial--boundary value problem on an open bounded domain. The maximum principle, the gradient 
estimate, and the H\"older norm estimate are obtained in order to show the existence of
solution by means of the Leray--Schauder theorem. 
Further, the diagonalization argument is employed to prove the existence of solution to the Cauchy problem. 
Remark that obtaining the gradient estimate is straightforward
and can be obtained from the similar result in \cite{lady} by freezing the non-local dependence $\tet_u$.
However, the estimate
of H\"older norms cannot be obtained in the similar manner because bounds for the time derivatives 
of the coefficients of \rf{PDE} have to be known in advance. 
It is well known, however, that
the H\"older norm estimates are crucial for application of the Leray--Schauder theorem and
the diagonalization argument.



The organization of the article is as follows. Section \ref{s2} is dedicated to the existence and
uniqueness of solution to abstract multidimensional non-local quasilinear parabolic PDEs of form \rf{PDE}.
We consider both the Cauchy problem and the initial--boundary value problem.
In Section \ref{s3}, we show that by means of formulas \rf{coeff}, the PIDE associated to FBSDE \rf{eqFBSDE} is included
into the class of non-local PDEs considered in Section \ref{s2}.
Finally,  in the same section we obtain 
 the existence and uniqueness theorem for FBSDEs with jumps and provide the
formulas connecting the solution to an FBSDE with the solution to the associated PIDE.

\section{Multidimensional non-local  quasilinear parabolic PDEs}
\label{s2}
In this section, we obtain the existence and uniqueness of solution for the initial--boundary value problem and the Cauchy problem
for  an abstract $\RR^m$-valued non-local quasilinear parabolic  PDE  of form  \rf{PDE},
where $\tet_u(t,x)$ is a function built via $u$, taking values in a normed space $E$,
and satisfying additional assumptions to be discussed later. Remark that the function $\tet_u$ considered
in this section is not necessary of the form specified in \rf{coeff}.

Let $\F\sub\Rnu^n$ be an open bounded domain with a piecewise-smooth boundary and non-zero interior angles.
For a more detailed description of the forementioned class of domains  we refer the reader to
\cite{lady} (p. 9). 
Further, in the case of the initial--boundary value problem we consider the boundary condition
\aaa{ 
u(t,x)= \psi(t,x), \quad (t,x) \in \{(0,T) \x \pl \F\} \cup  \big\{\{t=0\} \times \F \big\},
\label{initialsystem}
}
where $\psi$ is the boundary function defined as follows
\aaa{
\psi(t,x) =
\begin{cases} 
\ffi_0(x), \quad x\in \{t=0\} \times \F, \\
0, \quad (t,x) \in  [0,T] \x \pl \F.
\end{cases}
\label{boundaryfunc}
}
In the case of the Cauchy problem, we consider the initial condition
\aaa{
 u(0,x)= \ffi_0(x), \quad x \in \Rnu^n.
\label{cauchysystem}
}
Further, in the case of
the initial--boundary value problem, the coefficients of PDE \rf{PDE} are defined as follows:
$a_{ij}: [0,T] \x \F \x \RR^m \to \Rnu$, $a_i:  [0,T] \x \F \x \Rnu^m \x \Rnu^{m \x n} \x E\to\Rnu$, $i,j=1, \ldots, n$, 
$a:  [0,T] \x \F  \x \Rnu^m \x \Rnu^{m \x n} \x E\to\Rnu^m$.
 In the case of the Cauchy problem,  everywhere in the above definitions,
 $\F$ should be replaced with the entire space $\Rnu^n$.

We remark that due the presence of the function $\tet_u$,  the existence and uniqueness results of Ladyzenskaya et al \cite{lady} 
for initial--boundary value problem \rf{PDE}-\rf{initialsystem} and Cauchy problem
\rf{PDE}-\rf{cauchysystem} are not applicable to the present case.
\begin{rem}
\rm Without loss of generality we assume that $\{a_{ij}\}$ is a symmetric matrix. Indeed, since we are interested in $\C^{1,2}$-solutions
of \rf{PDE}, then for all $i,j$, $\pl^2_{x_ix_j}\! u = \pl^2_{x_j x_i}\! u$. Therefore,  $\{a_{ij}\}$ can be replaced with $\frac12(a_{ij} + a_{ji})$
for non-symmetric matrices. 
\end{rem}

\subsection{Steps involved in solving the problem}
In this subsection, we explain the scheme for obtaining the existence and
uniqueness theorem for non-local PDE \rf{PDE}. In each step, we mention whether it is 
an adaptation of the similar result in \cite{lady}, or the differences are essential.
\biz
\item[1.] {\it Maximum principle.} 
We start with the maximum principle for an
initial--boundary value
problem for PDE \rf{PDE} on a bounded domain. This step can be viewed as an 
adaptation of the similar result in  \cite{lady}.
\item[2.] {\it Gradient estimate.} To obtain an a priori estimate for the gradient of a classical
solution, we freeze the function $\tet_u$ in \rf{PDE}. After this, we are able to apply the result
from \cite{lady} on the gradient estimate.
\item[3.] {\it H\"older norm estimate.} This estimate 
cannot be obtained by simply freezing $\tet_u$, as in the previous step, since
to do so, we need bounds for the time derivatives of the coefficients of \rf{PDE}, where the letter will involve
the time derivative $\pl_t \tet_u$. To overcome this issue, we have to find an a priori bound for 
 $\pl_t u$.
At this step, the difference with \cite{lady} becomes essential since it requires to
estimate an $L_2$-type norm of $\tet_u$, given by \rf{coeff}, via the similar norm of $u$.
\item[4.] {\it Existence and uniqueness theorem for an initial--boundary value problem.}
The main tools in our proof are a priori estimates of H\"older norms obtained in the previous steps and 
the version of the Leray-Schauder theorem from \cite{ls}. 
The proof essentially relies on the scheme outlined in \cite{lady}.
\item[5.] {\it Existence theorem for a Cauchy problem.} 
The main technique here is the diagonalization argument, described in \cite{lady} for the
case of one equation, together with the existence theorem obtained in the previous step.
\item[6.] {\it Uniqueness theorem for a Cauchy problem.} To prove the uniqueness,
we use the results on fundamental solutions from the book \cite{friedman} along with
 Gronwall's inequality and estimates for $\tet_u$.
\ei

\subsection{Notation and terminology}
\lb{s1.1}

In this subsection we introduce the necessary notation that will be used throughout this article.

$T>0$ is a fixed real number, not necessarily small.

 $\F\sub \Rnu^n$ is an open bounded domain with a piecewise-smooth boundary $\pl \F$ and non-zero interior angles.

 
$\F_T = (0,T) \x \F$ and $\F_t=(0,t) \times  \F$, $t\in (0,T)$.
 
 $(\pl \F)_T = [0,T]\x \pl \F$ and  $(\pl \F)_t=[0,t] \x  \pl\F$, $t\in (0,T)$.
 
 $\ovl \F_T = [0,T] \x \ovl \F$, where $\ovl \F$ is the closure of $\F$.
 
  $\Gamma_t = (\{t=0\} \times \F) \cup (\pl \F)_t$,   $t\in [0,T]$.

 

 $(E, \|\cdot \|)$ is a normed space. 
 


 
  For a function  $\phi(t,x,u, p,w): [0,T] \x \ovl \F \x \Rnu^m \x \Rnu^{m \x n} \x E \to \Rnu^k$, where $k=1,2,\ldots$
  
 $\pl_x \phi$ or $\phi_x$ denotes the partial gradient with respect to $x\in\Rnu^n$;
 
 $\pl_{x_i} \phi$ or $\phi_{x_i}$ denotes the partial derivative $\frac{\pl}{\pl x_i} \phi$;

$\pl^2_{x_ix_j} \phi$ or $\phi_{x_i x_j}$ denotes the second partial derivative $\frac{\pl^2}{\pl x_i\pl x_j} \phi$;
 
  $\pl_t \phi$ or $\phi_t$ denotes the partial derivative $\frac{\pl}{\pl t} \phi$;
 
 $\pl_u \phi$ denotes the partial gradient with respect to $u\in\Rnu^m$;
 
 $\pl_{u_i} \phi$ or $\phi_{u_i}$ denotes the partial derivative $\frac{\pl}{\pl u_i}  \phi$; 
 
 $\pl_p \phi$ denotes  the partial gradient  with respect to $p\in \Rnu^{m \x n}$;
 
 $\pl_{p_i} \phi$ or $\phi_{p_i}$ denotes the partial gradient with respect to the $i$th column $p_i$ of 
 the matrix $p\in \Rnu^{m \x n}$;

 $\pl_w \phi$ denotes the partial G\^ateaux derivative of $\phi$ with respect to $w\in E$.

%
%
%
%
%
%
%
%
%
%
%
%
  
  $\hat\mu (s)$, $s\gt 0$, is a positive non-increasing continuous function.
 
 $\mu(s)$ and $\td \mu(s)$,  $s\gt 0$,   are positive non-decreasing continuous functions.
 
$P(s,r,t)$ and $\eps(s,r)$, $s,r,t\gt 0$,
are positive and non-decreasing with respect to each argument, 
whenever the other arguments are fixed.

  $\ffi_0(x)$ is the initial condition.
  
  $m$ is the number of equations in the system.
  
  $M$ is the a priori bound on $\ovl \F_T$ for the solution $u$ to problem \rf{PDE}-\rf{initialsystem} 
  (as defined in Remark \ref{rem-99}).
  
   $M_1$ is the a priori bound for $\pl_x u$ on $\ovl \F_T$.
  
  $\hat M$ is the a priori bound for $\|\tet_u\|_E$ on $\ovl \F_T$.
  
  $\mc K$ is the common bound for the partial derivatives and the H\"older constants, mentioned in Assumption (A8),
over the region $\ovl \F_T \x \{|u| \lt M \} \x \{ \| w\|_E \lt \hat M\} \x \{|p|\lt M_1 \}$, as defined in Remark \ref{rem8}.

 $K_{\xi,\zeta}$ is the constant defined in Assumption (A10).

The H\"older space $\C^{2+\beta}(\ovl \F)$,
 $\beta\in (0,1)$,
is understood as the (Banach) space with the norm
\aaa{
\lb{ho1}
\|\phi\|_{\C^{2 +\beta}(\ovl \F)} = \|\phi\|_{\C^2(\ovl \F)} + [\phi'']_\beta,
\quad \text{where} \quad
[\td \phi]_\beta = \sup_{\substack{x,y\in \ovl \F,\\ 0<|x-y|<1}}\frac{|\td\phi(x)- \td\phi(y)|}{|x-y|^\beta}.
}

For a function $\ffi(x,\xi)$ of more than one variable, the H\"older constant 
with respect to $x$ is defined as 
\aaa{
\lb{ho3}
[\ffi]^x_{\beta} =  \sup_{x,x'\in\ovl \F, \, 0<|x-x'|< 1}  \frac{|\ffi(x,\xi) - \ffi(x',\xi)|}{|x-x'|^\beta},
}
i.e., it is understood as a function of $\xi$.

The parabolic H\"older space $\C^{1+\frac{\beta}2,2+\beta}(\ovl \F_T)$, $\beta\in (0,1)$,
is defined as the Banach space of functions $u(t,x)$ possessing the finite norm
\mmm{
\lb{parabolic-holder-norm}
\|u\|_{\C^{1+\frac{\beta}2,2+\beta}(\ovl \F_T)} =
\|u\|_{\C^{1,2}(\ovl \F_T)} + \sup_{t\in [0,T]}[\pl_t u]_{\beta}^x + \sup_{t\in [0,T]}[\pl^2_{xx} u]_{\beta}^x \\
+ \sup_{x\in \ovl \F}\, [\pl_t u]_{\frac{\beta}2}^t + \sup_{x\in \ovl \F}\,[\pl_x u]_{\frac{1+\beta}2}^t
+ \sup_{x\in \ovl \F}\,[\pl^2_{xx} u]_{\frac{\beta}2}^t.
} 

$\C^{\frac{\beta}2, \beta}(\ovl \F_T)$, $\beta\in (0,1)$, denotes the space of functions $u\in \C(\ovl \F_T)$ 
possessing the finite norm
\aa{
\|u\|_{\C^{\frac{\beta}2,\beta}(\ovl \F_T)} =  \|u\|_{\C(\ovl \F_T)}+
\sup_{t\in [0,T]}[u]_{\beta}^x
+ \sup_{x\in \ovl \F}\, [u]_{\frac{\beta}2}^t.
}

$\C^{0,1}_0(\ovl \F_T)$ and $\C^{1,2}_0(\ovl \F_T)$ denotes the space of functions 
from $\C^{0,1}(\ovl \F_T)$ and $\C^{1,2}(\ovl \F_T)$, respectively, 
vanishing on $\pl \F$.

$\mc H(E,\Rnu^m)$ is the Banach space of bounded positively homogeneous maps $E\to \Rnu^m$ with the norm
$\|\phi\|_{\mc H} = \sup_{\{\|w\|_E \lt 1\}} |\phi(w)|$.

The H\"older space $\C^{2+\beta}_b(\Rnu^n)$,
 $\beta\in (0,1)$,
is understood as the (Banach) space with the norm
\aaa{
\lb{ho2}
\|\phi\|_{\C^{2 +\beta}_b(\Rnu^n)} = \|\phi\|_{\C^2_b(\Rnu^n)} + [\phi'']_\beta,
}
where $\C^2_b(\Rnu^n)$ denotes the space of twice continuously differentiable functions  on $\Rnu^n$ with bounded derivatives
up to the second order. The second term in \rf{ho2} is
the H\"older constant which is defined as in \rf{ho1} but the domain $\F$ has to be replaced with the entire space
$\Rnu^n$. 

Similarly, for a function $\ffi(x,\xi)$, $x\in\Rnu^n$, of more than one variable, the H\"older constant 
with respect to $x$ is defined as in \rf{ho3} but $\F$  should be replaced with $\Rnu^n$.

Further, the parabolic H\"older space $\C_b^{1+\frac{\beta}2,2+\beta}([0,T]\x \Rnu^n)$
is defined as the Banach space of functions $u(t,x)$ possessing the finite norm
\mm{
\|u\|_{\C_b^{1+\frac{\beta}2,2+\beta}([0,T]\x \Rnu^n)} =
\|u\|_{\C_b^{1,2}([0,T]\x \Rnu^n)} + 
\sup_{t\in [0,T]}[\pl_t u]_{\beta}^x + \sup_{t\in [0,T]}[\pl^2_{xx} u]_{\beta}^x \\
+ \sup_{x\in \Rnu^n}[\pl_t u]_{\frac{\beta}2}^t + \sup_{x\in \Rnu^n}[\pl_x u]_{\frac{1+\beta}2}^t
+ \sup_{x\in \Rnu^n}[\pl^2_{xx} u]_{\frac{\beta}2}^t,
} 
where
$\C_b^{1,2}([0,T]\x \Rnu^n)$ denotes the space of bounded continuous functions
whose mixed derivatives up to the second order in $x\in\Rnu^n$ and first 
order in $t\in [0,T]$ are bounded and continuous on $[0,T]\x \Rnu^n$.


We say that a smooth surface $S\sub \Rnu^n$ (or $S\sub [0,T]\x\Rnu^n$) is of class $\C^\gm$ 
(resp. $\C^{\gm_1,\gm_2}$),  where $\gm, \gm_1, \gm_2>1$ are not necessarily integers,
if at some local Cartesian coordinate system of each point $x\in S$, the surface $S$ is represented as a graph
of function of class $C^\gm$ (resp. $\C^{\gm_1,\gm_2}$).
For more details on surfaces of the classes $\C^\gm$ and $\C^{\gm_1,\gm_2}$,
we refer the reader to \cite{lady} (pp. 9--10). 

Furthermore, we say that a piecewise smooth surface $S\sub \Rnu^n$ is of class $\C^\gm$, $\gm>1$,
if  each of its smooth components is of this class.

%
%
%
%

The H\"older norm of a function $u$ on $\Gm_T$ is defined as follows
\aa{
\|u\|_{\C^{1+\frac{\beta}2,2+\beta}(\Gm_T)} = \max\Big\{\|u\|_{\C^{2+\beta}(\ovl\F)}, \|u\|_{\C^{1+\frac{\beta}2,2+\beta}((\pl \F)_T)}\Big\},
}
where the norm $\|u\|_{\C^{1+\frac{\beta}2,2+\beta}((\pl \F)_T)}$ is defined in \cite{lady} (p. 10). However,
since we restrict our consideration only to functions vanishing on the boundary $\pl \F$, we do not need the details 
on the definition of H\"older norms on $(\pl \F)_T$, i.e., in our case it always holds that
\aa{
\|u\|_{\C^{1+\frac{\beta}2,2+\beta}(\Gm_T)} = \|u\|_{\C^{2+\beta}(\ovl \F)}.
}
  
 \begin{rem}
\rm
Some notation of this article is different than in the book of Ladyzhenskaya et al. \cite{lady}. For reader's convenience, we
provide the correspondence of the most important notation:
$\Om = \F, \; S= \pl \F, \; S_T = (\pl \F)_T, \; Q_T = \F_T, \; \Gm_T = \Gm_T, \; N=m$.
\end{rem} 
  
\subsection{Maximum principle}
\label{maximum-principle}
In this subsection, we obtain the maximum principle for problem \rf{PDE}-\rf{initialsystem} under
assumptions (A1)--(A4) below. Obtaining an a priori bound for the solution to problem \rf{PDE}-\rf{initialsystem}
is an essential step for obtaining other a priori bounds, as well as proving the existence of solution.

We agree that the functions $\mu(s)$ and $\hat \mu(s)$ in the assumptions below 
are non-decreasing and, respectively, non-increasing,
continuous,  defined for positive arguments, and
taking positive values. Further, 
$L_E$, $c_1$, $c_2$, $c_3$ are non-negative constants.
Assume the following.
\bi
\item[\bf (A1)] For all $(t,x,u)\in \ovl \F_T \x \RR^m$ and 
$\xi = (\xi_1, \ldots, \xi_n)\in\RR^n$,
\aa{
\hat\mu(|u|) |\xi|^2 \lt \sum_{i,j=1}^n a_{ij}(t,x,{u}) \xi_i \xi_j \lt \mu(|u|) |\xi|^2.
}
\item[\bf (A2)]
The function $\tet_u: \ovl \F_T \to E$ is defined for each $u\in \C^{0,1}_0(\ovl \F_T)$, and
such that $\sup_{\ovl \F_T}\|e^{-\la t }\tet_u(t,x)\|_E \lt L_E \sup_{\ovl \F_T} |e^{-\la t } u(t,x)|$
for all $\la\gt 0$.
\item[\bf (A3)]
There exists a function $\zeta:  \mc R_0 \x \Rnu^n \to [0,\infty)$, 
where $\mc R_0 = \ovl \F_T \x \Rnu^m \x \Rnu^{m \x n} \x E$,
such that for all $(t,x,u,p,w) \in  \mc R_0$,
$\zeta(t,x,u,p,w,0) = 0$  and
\aa{
\big(a(t,x,u,p,w), \, u\big) \gt -c_1 - c_2 |u|^2 - c_3\|w\|_E^2 - \zeta(t,x,u,p,w,p^{_\top} u).
}
 \item[\bf (A4)] The function $\ffi_0: \ovl \F \to \Rnu^m$ is of class $\C^{2+\beta}(\ovl\F)$ with $\beta\in (0,1)$.
\ei
\begin{lem} 
\lb{lem101}
Assume (A1).
If a twice continuously differentiable function $\ffi(x)$ achieves a local maximum at $x_0 \in \F$, then for any $(t,u)\in [0,T]\x \RR^m$,
\aa{
\sum\nolimits_{i,j} a_{ij}(t,x_0,u) \ffi_{x_ix_j}(x_0) \lt 0.
}
\end{lem}
\begin{proof}
For each $(t,u) \in [0,T] \x \RR^m,$ we have
\aa{
\sum_{i,j=1}^n a_{ij}(t,x_0,u)  \ffi_{x_i x_j}(x_0)= \sum_{i,j,k,l=1}^n \ffi_{y_k y_l}(x_0) a_{ij}(t,x_0,u)  v_{ik} v_{jl} = \sum_{k=1}^n \la_k \ffi_{y_k y_k}(x_0), 
} 
where $\{v_{ij}\}$ is the matrix whose columns are the  vectors of the orthonormal eigenbasis of $\{a_{ij}(t,x_0,u)\}$,
$(y_1,\cdots,y_n)$ are the coordinates with respect to this eigenbasis, and $(\la_1,\cdots,\la_n)$ are the eigenvalues of $\{a_{ij}(t,x_0,u)\}$.

Note that by  (A1), $\la_k=\sum_{i,j =1}^n  a_{ij} v_{ik} v_{jk} \gt \hat\mu(|u|)  > 0$. Let us show that $\ffi_{y_k y_k}(x_0) \lt 0$.  
Since  $\ffi(y_1, \ldots, y_n)$ has a local  maximum at $x_0,$ then $\ffi_{y_k}(x_0) = 0$ for all $k$. Suppose for an
arbitrary fixed $k$, $\ffi_{y_ky_k}(x_0)>0$. Then,
by the second derivative test, the function $\ffi(y_1, \ldots, y_n)$, considered
as  a function of $y_k$ while the rest of the variables is fixed, would have a local minimum at $x_0$. The latter is not the case. Therefore,
$\ffi_{y_ky_k}(x_0)\lt 0$. The lemma is proved.
\end{proof}
Lemma \ref{amp} below will be useful.
\begin{lem}
\lb{amp}
For a function $\ffi: \F_T \to \Rnu$, 
one of the mutually exclusive conditions 1)--3) necessarily holds:
\bi
\item[1)] $\sup_{\F_T} \ffi(t,x) \lt 0$;
\item[2)] $0 < \sup_{\F_T}  \ffi(t,x) = \sup_{\Gm_T}  \ffi(t,x)$;
\item[3)] $\exists \, (t_0,x_0) \in (0,T] \times \F$ such that $\ffi(t_0,x_0) =  \sup_{\F_T}  \ffi(t,x) > 0$.
\ei
\end{lem}
\begin{proof}
The proof is straightforward.
\end{proof}

\begin{thm} [Maximum principle for initial--boundary value problem \rf{PDE}-\rf{initialsystem}]
\lb{max-princ}
Assume (A1)--(A4). 
If $u(t,x)$ is a $\C^{1,2}(\ovl \F_T)$-solution to  problem \rf{PDE}-\rf{initialsystem}, then
\aaa{
\lb{999}
\sup_{\ovl\F_T} |u(t,x)| \lt  e^{\la T} \max \big\{ \sup_{\ovl\F} |\ffi_0(x)|, \, \sqrt{c_1} \big\}, \quad \text{where} \quad \la  =  c_2  +  c_3 L^2_E + 1.
}
\label{maxminsys}
\end{thm}
\begin{proof}
Let $v(t,x)=u(t,x) e^{-\la t}.$ 
Then, $v$ satisfies the equation
\mm{
 - \sum_{i,j=1}^{n} a_{ij} (t,x,{u}){v}_{x_i x_j} +e^{-\la t} a(t,x,u,u_x,\tet_u)+  \sum_{i=1}^{n} a_{i}(t,x,u,u_x, \tet_u)v_{x_i} + \lambda {v} + {v}_t =0.
 }
 Multiplying the above identity scalarly by $v$, 
 and noting that $(v_{x_ix_j},v) = \frac12 \pl^2_{x_i x_j} |v|^2 - (v_{x_i},v_{x_j})$,
 we obtain
\mmm{
 - \frac12 \sum_{i,j=1}^{n}  a_{ij} (t,x,{u})\pl^2_{x_i x_j} |v|^2+ 
  e^{- \lambda t} (a(t,x,u,u_x, \tet_u),v)
 \\ + \sum_{i,j=1}^{n}  a_{ij} (t,x,{u})  (v_{x_i}, {v}_{x_j}) +
 \frac12  \sum_{i=1}^n a_{i}(t,x,u,u_x, \tet_u) \pl_{x_i} |v|^2 + \lambda |v|^2  
 + \frac12 \pl_t |v|^2=0,
  \label{mixmaxsys}
 }
 where $u$ and $v$ are evaluated at $(t,x)$.
If $t=0$, then \rf{999} follows trivially. 
Otherwise, for the function $w=|v|^2$, one of the conditions 1)--3) of Lemma \ref{amp} necessarily holds.
Note that condition 1) is excluded. Furthermore, if 2) holds, then 
\aaa{
\lb{9976}
\sup_{\ovl\F_T}|u(t,x)|  \lt e^{\la T} \sup_{\ovl\F_T}|v(t,x)| \lt e^{\la T}\sup_{\ovl\F} |\ffi_0(x)|.
}
Suppose now that 3) holds, i.e., 
the maximum of $|v|^2$   is achieved at some point 
$(t_0,x_0)\in (0,T] \x \F$.
Then, we have  
\aaa{
\lb{789}
\pl_x w(t_0,x_0)=0 \quad \text{and} \quad  \pl_t w(t_0,x_0) \gt 0.
}
By Lemma \ref{lem101}, the first term in \rf{mixmaxsys} is non-negative at $(t_0,x_0)$.
Further, assumption (A1) and identities \rf{789} imply that the third, fourth, and the last term on the left-hand side of   
\rf{mixmaxsys}, evaluated at $(t_0,x_0)$, are non-negative.
Consequently, substituting $v(t_0,x_0)=u(t_0,x_0) e^{-\la t_0}$, we obtain 
\aaa{
\lb{eq23}
e^{-2\la t_0} \big(a(t_0,x_0, u_x(t_0,x_0), \tet_u(t_0,x_0)), u(t_0,x_0)\big)
+ \la |v(t_0,x_0)|^2 \lt 0.
}
Further, we remark that $v_x(t_0,x_0)^{\top} v(t_0,x_0) = \frac12 w_x(t_0,x_0) = 0$. Therefore, by (A2),
\mmm{
\lb{es2}
0\gt e^{-2\la t_0}\big(a(t_0,x_0,u(t_0,x_0), u_x(t_0,x_0),  \tet_u(t_0,x_0)), u(t_0,x_0)\big) + \la |v(t_0,x_0)|^2 \\
\gt -c_1e^{-2\la t_0} - c_2|v(t_0,x_0)|^2 - c_3  \|e^{-\la t_0} \tet_u(t_0,x_0)\|_E^2 
- \zeta(\ldots,0) 
+ \la |v(t_0,x_0)|^2 \\
\gt  -c_1 -c_2 |v(t_0,x_0)|^2 - c_3 L_E^2 |v(t_0,x_0)|^2 + \la |v(t_0,x_0)|^2. 
}
Picking $\la  = c_2  + c_3 L^2_{E}  + 1$, we obtain that $|v(t_0,x_0)|^2 \lt c_1$. Therefore,
\aa{ 
\sup_{\ovl\F_T}|u(t,x)| \lt \sqrt{c_1} e^{\la T}.
}
The above inequality together with \rf{9976} implies \rf{999}. 
\end{proof}
\begin{cor}
\lb{cor91}
Let assumptions of Theorem \ref{max-princ} hold. 
If, in (A3), $c_1 = 0$, then
\aaa{
\lb{991}
\sup_{\ovl\F_T} |u(t,x)| \lt  e^{\la T} \sup_{\ovl\F} |\ffi_0(x)|  \quad \text{with} \quad \la  = c_2  +  c_3 L^2_E  +1.
}
If $c_1>0$, then
\aaa{
\lb{992}
\sup_{\ovl\F_T} |u(t,x)| \lt  e^{(\la+ c_1) T} \max\{\sup_{\ovl\F} |\ffi_0(x)|, 1\}.  
}
\end{cor}
\begin{rem}
\lb{rem-99}
\rm
Let $M$ denote the biggest of right-hand sides of \rf{999} and \rf{992}. By Theorem \ref{maxminsys}
and Corollary \ref{cor91}, $M$ is an a priori bound  for $|u(t,x)|$ on $\ovl \F_T$. 
 Everywhere below throughout the text, by $M$ we understand the quantity defined above. 
Furthermore, by (A2), $L_E M$ is a bound for $\|\tet_u(t,x)\|_E$, which we denote by $\hat M$.
\end{rem}
\begin{rem}
\rm
 The function $\zeta$ was added to the right-hand side of the inequality in (A3)  just for
 the sake of generality (it is not present in the similar assumption in \cite{lady}). Indeed, the presence of
 this function does not give any extra work in the proof.
\end{rem}
\subsection{Gradient estimate}
We now formulate assumptions
(A5)--(A7) which, together with previously introduced assumptions (A1)--(A4),
will be necessary for obtaining an a priori bound for the gradient $\pl_x u$ of the solution
$u$ to problem \rf{PDE}-\rf{initialsystem}. 
Obtaining the gradient estimate is crucial for obtaining estimates of H\"older norms,
as well as for the proof of existence. 
Everywhere below, $\mc R$ and $\mc R_1$ are regions defined as follows
\aaa{
\mc R = \ovl \F_T \x \{ |u| \lt M\}  \x \Rnu^{m \x n} \x  \{\|w\|_E \lt \hat M\};
\quad
\mc R_1 = \ovl\F_T \x \{ |u| \lt M\}.
\lb{regions}
}
Further,  the functions $\td \mu(s)$, $\eta(s,r)$, $P(s,r,t)$, and $\eps(s,r)$ in the assumptions below 
are continuous,  defined for positive arguments,
taking positive values, and non-decreasing with respect to each argument, 
whenever the other arguments are fixed.

Assumptions (A5)--(A7) read:
\bi
 \item[\bf (A5)]  For all $(t,x,u,p,w)\in \mc R$ it holds that 
   \aa{
   & |a_i (t,x,u,p,w)| \leq \eta(|u|, \|w\|_E)(1+|p|), \quad i\in \{1,\ldots, n\}, \quad \text{and}\\
   & |a(t,x,u,p, w)| \lt \big(\eps(|u|,  \|w\|_E) + P(|u|, |p|,  \|w\|_E)\big)(1+|p|)^2,
}
where $\lim_{r\to \infty} P(s, r, q) = 0$ and
$2(M+1) \eps(M,\hat M) \lt \hat \mu(M)$.
\item[ \bf (A6)] $a_{ij}$, $a$ and $a_i$ are continuous on $\mc R$;
 $\pl_x a_{ij}$ and $\pl_u a_{ij}$ exist and are continuous on $\mc R_1$;
moreover, 
$\max \big\{\big| \pl_x a_{ij}(t,x,u)\big| , \big|\pl_u a_{ij}(t,x,u)\big| \big\}\lt \td \mu(|u|)$.
\item[\bf (A7)] The boundary $\pl \F$ is of class $\C^{2+\beta}$.
\ei
In Theorem \ref{grad-est-th} below, we obtain the gradient estimate 
for a $\C^{1,2}(\ovl \F_T)$-solution $u(t,x)$ of problem \rf{PDE}-\rf{initialsystem}.
The main idea is to freeze $\tet_u$ in the coefficients $a_i$ and $a$ and 
apply the result of \cite{lady} on the gradient estimate of a classical solution to
a system of quasilinear parabolic PDEs. 
\begin{thm}(Gradient estimate)
\lb{grad-est-th}
 Let (A1)--(A7) hold, and let
$u(t,x)$  be a $\C^{1,2}(\ovl \F_T)$-solution to problem  \rf{PDE}--\rf{initialsystem}.
Further  let $M$ be the a priori bound for $|u(t,x)|$ on $\ovl \F_T$ whose existence 
was established by Theorem \ref{maxminsys}.
Then, there exists a constant $M_1>0$, depending only 
on $M$, $\hat M$, $\sup_{\ovl\F} |\pl_x \ffi_0|$,
$\mu(M)$,  $\hat\mu(M)$, $\td\mu(M)$, $\eta(M,\hat M)$, $\sup_{q\gt 0} P(M,q,\hat M)$, and $\eps(M,\hat M)$
such that
 \aaa{
 \lb{grad-est}
\sup_{\ovl \F_T} |\pl_x u| \lt M_1.
}
\end{thm}
\begin{proof}
In \rf{PDE}, we freeze $\tet_u$  in the coefficients $a_i$ and $a$.
 Non-local PDE \rf{PDE} is, therefore, reduced  to the following quasilinear parabolic PDE with respect to $v$
\mmm{
\lb{PDE-lady}
- \sum_{i,j=1}^{n} a_{ij} (t,x,v)\pl^2_{x_i x_j}v + \sum_{i=1}^n a_i (t,x,v, \pl_{x} v, \tet_u(t,x))\pl_{x_i} v \\ 
+ a(t,x,v, \pl_x v, \tet_u(t,x))  +  \pl_t v=0
}
with initial--boundary condition \rf{initialsystem}.
Since $\hat M$ is an a priori bound for $\|\tet_u(t,x)\|_E$ (see Remark \ref{rem-99}), we are in the assumptions of Theorem 6.1 from  \cite{lady} (p. 592) on the gradient estimate
for solutions of PDEs of form \rf{PDE-lady}. Indeed, assumptions (A1) and (A5) are the same as in Theorem 6.1,
and  (A6) immediately implies the continuity of functions $(t,x,v,p) \to a(t,x,v, p, \tet_u(t,x))$ and $(t,x,v,p) \to a_i(t,x,v, p, \tet_u(t,x))$ 
in the region $\ovl \F_T \x \{ |v| \lt M\} \x \Rnu^{m \x n}$.
Further, (A5) implies conditions (6.3) on p. 588 and inequality (6.7) on p. 590 of \cite{lady}.
It remains to note that by (A3),
\aa{
\big(a(t,x,v, p, \tet_u(t,x)), v)  \gt -c_1' - c_2|v|^2 + \zeta(t,x,v,p,\tet_u(t,x), p^\top v),
}
where $c_1'=c_1 + c_3 \hat M^2$.
Therefore,
by Corollary \ref{cor91}, any solution $v(t,x)$ of \rf{PDE-lady} satisfies the estimate
$\sup_{\ovl \F_T} |v(t,x)| \lt  e^{(c_2+1+c_1') T} \max \big\{ \sup_{\ovl\F} |\ffi_0(x)|, 
1 \big\}\lt M$.

Since $v(t,x) = u(t,x)$ is a $\C^{1,2}(\ovl \F_T)$-solution to \rf{PDE-lady}, then by Theorem 6.1 of \cite{lady}, 
estimate \rf{grad-est} holds true. By the same theorem, the constant
$M_1$ only depends on $M$,  $\sup_{\ovl \F} |\pl_x \ffi_0|$, 
$\mu(M)$,  $\hat\mu(M)$, $\td\mu(M)$, $\eta(M,\hat M)$, $\sup_{q\gt 0} P(M,q,\hat M)$, and $\eps(M,\hat M)$.
\end{proof}
\subsection{Estimate of $\pl_t u$}
\lb{ut-est}
Now we complete the set of assumptions (A1)--(A7) with assumptions (A8)--(A10) below. 
All together, these assumptions
are necessary to obtain an a priori bound for the time derivative $\pl_t u$ which is crucial for proving that
any $\C^{1,2}(\ovl \F_T)$-solution  to problem \rf{PDE}--\rf{initialsystem}
belongs to class $\C^{1+\frac{\beta}2,1+\beta}(\ovl \F_T)$ and obtaining a
 bound for the $\C^{1+\frac{\beta}2,1+\beta}(\ovl \F_T)$-norm of this solution.
 The region  $\mc R_1$ is defined, as before,  by \rf{regions},
and the region $\mc R_2$ is defined as follows
\aaa{
\lb{r2}
\mc R_2 = \ovl \F_T \x \{|u| \lt M \}  \x \{|p|\lt M_1 \} \x \{\|w\|_E \lt \hat M\}.
}
Assumptions (A8)--(A10) read:
\bi
\item[\bf (A8)] $\pl_t a_{ij}$, $\pl^2_{uu} a_{ij}$, $\pl^2_{ux} a_{ij}$, $\pl^2_{xt} a_{ij}$, 
$\pl^2_{ut} a_{ij}$ exist and are continuous on $\mc R_1$; 
$\pl_t a$, $\pl_u a$, $\pl_p a$, $\pl_w a$, $\pl_t a_i$, $\pl_u a_i$, $\pl_p a_i$, $\pl_w a_i$ exist and
are continuous and bounded on $\mc R_2$; $a$ and $a_i$ are  
$\beta$-H\"older continuous in $x$, $\beta\in (0,1)$, and locally Lipschitz in $w$ with the H\"older and Lipschitz constants 
 bounded over $\mc R_2$.
\item[\bf (A9)] For each $u\in \C^{1,2}_0(\ovl \F_T)$, $\pl_t \tet_u$ and $\pl_x \tet_u$ exist and
are continuous and bounded; 
moreover, the bounds for $\|\pl_t \tet_u\|$ and $\|\pl_x \tet_u\|$ only depend
on the bounds for $|\pl_t u(t,x)|$ and $|\pl_x u(t,x)|$ in $\ovl\F_T$.
\item[\bf (A10)]  For all $u\in \C^{1,2}_0(\ovl \F_T)$, $(t,x)\in \ovl\F_T$, it holds that 
\aaa{
 \hspace{1.2cm}   
 \frac{\tet_u(t+\Dl t,x) - \tet_u(t,x)}{\Dl t} = \hat \tet_v(t,x) + \zeta_{u,u_x}(t,x) v(t,x) + 
\xi_{u,u_x}(t,x),
\lb{repa10}
}
where $v(t,x)=(\Dl t)^{-1} \big(u(t+\Dl t,x) - u(t,x)\big)$, $\zeta_{u,u_x}$, $\xi_{u,u_x}$ are bounded functions with values in  
$\mc L(\Rnu^m,E)$ and $E$, respectively,  depending non-locally on $u$ and $u_x$ (their common
bound will be denoted by $K_{\xi\zeta}$),
and $\hat \tet_v: \ovl \F_T \to E$, defined for each $v\in \C^{1,2}_0(\ovl \F_T)$, is such that
for all $\al>0$ and $\tau\in (0,T)$,
\aaa{
\lb{i-w}
 \int_{\F^\al_\tau(|v|^2)} \|\hat \tet_v(t,x)\|^4_E \, dt\, dx \lt \hat L_E \Big(\int_{\F^\al_\tau(|v|^2)} |v(t,x)|^4 dt dx + \al^2 \la(\F^\al_\tau)\Big),
}
where $\hat L_E>0$ is a constant depending on $\|u\|_{\C^{0,1}(\ovl\F_T)}$, $\F^\al_\tau(|v|^2) = \{(t,x)\in\F_\tau: |v(t,x)|^2>\al\}$, and $\la$ is the Lebesgue measure on $\Rnu^{n+1}$.
  \ei
  \begin{rem}
\rm
\lb{rem8}
The common bound over $\mc R_2$ for the partial derivatives and the H\"older constants
mentioned in assumption (A8) and related to 
the functions $a$ and $a_i$ will be denoted by $\mc K$. 
\end{rem}
\begin{rem}
\rm
According to the results of \cite{shapiro} (p. 484), for locally Lipschitz mappings in normed spaces,
the G\^ateaux and Hadamard directional differentiabilities are equivalent. Moreover, 
the local Lipschitz constant of a function is the same as the global Lipschitz constant 
of its G\^ateaux derivative. Thus, under (A8), the chain rule holds
for the G\^ateaux derivatives $\pl_w a$ and $\pl_w a_i$ which, 
moreover, are globally Lipschitz and positively homogeneous. 
\end{rem}
The following 
maximum principle for non-local linear-like parabolic PDEs written in the divergence form, is crucial 
for obtaining the a priori bound for $\pl_t u$.

Consider the following system of non-local PDEs in the divergence form
 \mmm{
\pl_t u - \sum_{i=1}^n \pl_{x_i}  \Big[\sum_{j=1}^n \hat a_{ij}(t,x) \pl_{x_j} u + A_i(t,x)u + f_i(t,x)\Big]
+ \sum_{i=1}^n B_i(t,x) \pl_{x_i} u \\
 + A(t,x)u +C(t,x)\big(\hat \tet_u(t,x) \big)+ f(t,x) = 0, 
\qquad 
 u(0) = u_0,
 \label{linearsys}
}
where $\hat a_{ij}: \ovl \F_T \to \Rnu$, $A_i:  \ovl \F_T \to \Rnu^{m \x m}$,  $B_i: \ovl \F_T \to \Rnu^{m \x m}$,  
$f_i:  \ovl \F_T  \to \Rnu^m,$ $i,j=1,\ldots, n$,  $A:  \ovl \F_T  \to \RR^{m \x m}$,     $f: \ovl \F_T  \to \Rnu^m$,
and $C: \ovl \F_T \to \mc H(E,\Rnu^m)$,
where $\mc H(E,\Rnu^m)$ is the Banach space of bounded positively homogeneous maps $E\to \Rnu^m$ with the norm
$\|\phi\|_{\mc H} = \sup_{\{\|w\|_E \lt 1\}} |\phi(w)|$.
In \rf{linearsys}, the function $u$ together with its partial derivatives, as usual, is evaluated at $(t,x)$ and $\hat \tet_u(t,x)$
is an $E$-valued function built via $u$ and satisfying inequality \rf{i-w}. Remark that all terms in
\rf{linearsys}, except the term containing $\hat \tet_u(t,x)$, are linear in $u$.

The lemma below, which is a version of the integration-by-parts formula, can be found in \cite{lady} (p. 60).
\begin{lem}
\lb{lem25}
Let $f$ and $g$ be real-valued functions from the Sobolev spaces $W^{1,p}(\mathbb G)$ and $W^{1,q}(\mathbb G)$ 
{\rm (}$\frac1p + \frac1q  \lt 1+ \frac1n${\rm )}, respectively,
where $\mathbb G\sub \Rnu^n$ is a bounded domain. Assume that the boundary $\pl \mathbb G$ is piecewise smooth and
that  $fg = 0$ on $\pl \mathbb G$. Then, 
\aa{
\int_{\mathbb G} f \, \pl_{x_i} g \, dx = - \int_{\mathbb G} g \, \pl_{x_i} f \, dx.
}
\end{lem} 
Further, for each $\tau,\tau' \in [0,T]$, $\tau<\tau'$, we define the squared norm
\aaa{
\lb{norm}
\|v\|^2_{\tau,\tau'}= \sup_{t \in [\tau,\tau']} \|v^2(t,\fdot)\|_{L_2(\F)}^2 + \|\pl_x v\|^2_{L_2(\F_{\tau,\tau'})},
}
where $\F_{\tau,\tau'} =  \F \x [\tau,\tau']$.
Furthermore, for an arbitrary real-valued function $\phi$ on  $\ovl\F_T$ and a number $\al>0$, we define $\phi^\al = (\phi-\al)^+$
and $\F^\al_\tau(\phi) = \{(t,x)\in \F_\tau: \phi>\al\}$, where $\tau\in (0,T]$.
The following result was obtained in \cite{lady} (Theorem 6.1, p. 102). It will be 
used in Lemma \ref{lem99h}.
\begin{pro}
\lb{prop16} 
Let $\phi(t,x)$ be  a real-valued function of class  $\C(\ovl\F_\tau)$ such that
$\sup_{(\pl \F)_\tau}\phi \lt  \hat \al,$ where $\hat \al \gt 0$. Assume  for all $\al \gt \hat \al$ and 
for a positive constant $\gm$, it holds that 
$\|\phi^\al\|_{0,\tau} \lt \gm \al \,\sqrt{\la_{n+1}(\F^\al_\tau(\phi))}$,
where $\la_{n+1}$ is the Lebesgue measure on $\Rnu^{n+1}$. 
Then, there exists a constant $\dl>0$, depending only on $n$, such that
 \aa{
 \sup_{\ovl\F_\tau} \phi(t,x) \lt 2\,\hat \al\, \big(1+ \dl\, \gm^2 \, \tau \, \la_n(\F) \big).
 }
\end{pro}
\begin{rem}
\rm
We attributed the values $1+\cp = r=q = 4$ for the space dimensions $n=1,2$
 and $1+\cp = r=q = \frac{2n}{n-1}$ for $n\gt 3$ to the constants
  $r$, $q$, and $\cp$ appearing
in the original version of Theorem 6.1 in \cite{lady} (p. 102),
 since for our application we do not need Theorem 6.1 in the most general form.
Also, we remark that by our choice of the parameters, $1+\frac1\cp < 2$ for all space dimensions $n$.
\end{rem}
\begin{lem}
\lb{lem99h}
Assume the coefficients $\hat a_{ij}$, $A_i$, $B_i$, $f_i$, $f$, $A$, and $C$ are of class $\C(\ovl\F_T)$
 and that $\sum_{i,j=1}^n \hat a_{ij}(t,x)\xi_i\xi_j \gt \ro \|\xi\|^2$ for all $(t,x)\in\ovl\F_T$, $\xi\in\Rnu^m$,
 and for some constant $\ro>0$. 
Let $u(t,x)$ be a generalized solution to problem \rf{linearsys} which is of class $\C^{1,1}(\ovl \F_T)$
and such that $\hat \tet_u$ satisfies \rf{i-w}.
 Further let $v=|u|^2$.
Then, there exist a number $\tau\in (0,T]$ and a constant $\gm>0$, where $\tau$ depends on
 the common bound $\mc A$ over $\ovl\F_T$ for the coefficients $A_i$, $B_i$, $f_i$, $f$, $A$, $C$,
 and also on $\hat L_E$, $\ro$, $n$, and $\la_n(\F)$, and 
$\gm$ depends on the same quantities as $\tau$ and on $\sup_{\ovl \F} |u_0|$,  such that
\aaa{
\lb{boundC}
\|v^\al \|_{0,\tau} \lt \gm\, \al\, \sqrt{\la_{n+1}(\F^\al_\tau(v))} \quad \text{for all} \;\,  \al \gt \sup_{\ovl\F}|u_0|^2 + 1.
}
\end{lem}
\begin{proof}
Let $\tau\in (0,T]$.
 Multiplying PDE \rf{linearsys} scalarly by a $W^{1,p}(\ovl \F_\tau)$-function $\eta(t,x)$ ($p>1$)
vanishing on $\pl \F_\tau$ and applying the integration-by-parts formula (Lemma \ref{lem25}), we obtain
\mmm{
\lb{eq56}
 \int_{\F_\tau} \Big[(u_t(t,x), \eta(t,x) )+
 \sum_{i=1}^n  \Big( \sum_{j=1}^n \hat a_{ij}(t,x) u_{x_j} + A_i(t,x) u + f_i(t,x), \eta_{x_i}(t,x) \Big) \\+ 
\Big(\sum_{i=1}^nB_i(t,x) u_{x_i} + A(t,x)u + f(t,x) + C(t,x)\big(\hat \tet_u (t,x)\big),\eta(t,x) \Big)\Big] \, dt dx=0.
}
For simplicity of notation, we write $\F^\al_\tau$ for $\F^\al_\tau(v)$.
Define $\eta(t,x)=2u(t,x)v^\al(t,x)$
and note that $v^\al$ and its derivatives vanish outside of $\F_\tau^{\al}$.
Since
 $(\pl_t u ,\eta)=2(\pl_t u, u)v^\al=(\pl_t v) v^\al= \pl_t(v^\al) v^\al=\frac12 \pl_t(v^\al)^2$,
 we rewrite \rf{eq56} as follows
\aaa{
\frac12 \int_\F (v^\al)^2  \Big|_0^{\tau} dx & +
2\int_{\F_{\tau}^\al}\Big[ \Big( \sum_{i=1}^n\Big(\sum_{j=1}^n \hat a_{ij}(t,x) u_{x_j} + 
A_i(t,x) u + f_i(t,x)\Big) ,  (uv^\al)_{x_i} \Big) \notag\\
&+ 2\Big( \sum_{i=1}^nB_i(t,x) u_{x_i} + A(t,x)u + C(t,x)(\hat \tet_u)  + f(t,x) ,  uv^\al \Big)\Big] \, dt dx=0.
\lb{eq57}
}
Note that the following inequalities hold on ${\F_{\tau}^{\al}}$:
\aa{
&\mml{
2\sum_{i,j=1}^n \hat a_{ij}(t,x) (u_{x_j}, (uv^\al)_{x_i}) = 2\sum_{i,j=1}^n \hat a_{ij}(t,x) (u_{x_i}, u_{x_j }) v^\al + \sum_{i,j=1}^n  
\hat a_{ij}(t,x) v_{x_j} v^\al _{x_i} \\   \gt 2\ro |u_x|^2 v^\al + \ro(v^\al_x)^2;
}\\
&\mml{
 2 ( A_i u , (uv^\al)_{x_i}) \lt 2 |A_i| (|u| |u_{x_i}| v^\al + v |v^\al_{x_i}|) \lt \frac1\epsilon |A_i|^2 v v^\al + 
 \frac1\epsilon |A_i|^2 v^2 \\   \hspace{3.5cm} + \epsilon v^\al |u_{x_i}|^2 + \epsilon  |v^\al_{x_i}|^2
 \lt \frac{2}{\epsilon}|A_i|^2 v^2  + \epsilon v^\al |u_{x_i}|^2 + \epsilon  |v^\al_{x_i}|^2;
  }\\
 &\mml{
2 ( f_i, (uv^\al)_{x_i}) \lt 2 |f_i| ( |u_{x_i}| v^\al + |u| |v^\al_{x_i}|)  \lt \frac1\epsilon |f_i|^2 (v^\al +  v)
+ \epsilon\big[ v^\al |u_{x_i}| ^2+  |v^\al_{x_i}|^2\big];
}\\
&\mml{
 2(B_i u_{x_i} , uv^\al) \lt \frac1\epsilon |B_i|^2  v v^\al + \epsilon  |u_{x_i}|^2v^\al;  \qquad 2(Au, uv^\al) \lt  2|A| v^2;
}\\
&\mml{
 2(f, uv^\al)  \lt  2   |f| v^{\frac32}  \lt 2 |f|(1+v^2);
}\\
&\mml{
2\hs \int_{\F^\al_\tau} \hsp (C (\hat \tet_u), uv^\al) dt dx \lt \mc A \int_{\F^\al_\tau} \hsp 
\big(\|\hat\tet_u\|^4_E + v^2 + (v^\al)^2\big) dt dx 
\lt \hat {\mc A}  \Big[ \int_{\F^\al_\tau} \hs v^2 dt dx + \al^2 \la(\F^\al_\tau)\Big],
}
}
where $\epsilon>0$ is a small constant and the last inequality holds by \rf{i-w} with
$\hat {\mc A}$ being a constant that depends only on $\mc A$ and the constant $\hat L_E$ from \rf{i-w}.  
By virtue of these inequalities, from \rf{eq57} we obtain
\mmm{
\lb{456k}
\frac12 \int_\F (v^\al(\tau,x))^2 \,dx + \ro \int_{\F_\tau^{\al}} \{2 |u_x|^2 v^\al + (v^\al_x)^2 \} dxdt \lt   \frac12 \int_\F v^\al(0,x)^2 \,dx \\
+ \int_{\F_{\tau}^{\al}} \big( \td{\mc A}_\epsilon (1+v^2) 
+  3\epsilon |u_x|^2v^\al + 2\epsilon |v^\al_x|^2 \big) \, dt dx + \hat{\mc A}\, \al^2 \la(\F^\al_\tau),
}
where $\mc{\td A}_\epsilon = \epsilon^{-1} \sup_{\ovl\F_\tau}\big( 2 \sum_{i=1}^n|A_i|^2 +  
\sum_{i=1}^n|f_i|^2  + \sum_{i=1}^n|B_i|^2  + \epsilon |A| + \epsilon |f|   \big)$.
Picking $\epsilon=\frac{\ro}4$ and defining  $\td \ro = \min(\frac12,\frac{\ro}{2})$, 
for $\al\gt \sup_{\ovl\F} |u_0|^2 + 1$, \rf{norm} and \rf{456k} imply
\mmm{
\lb{456h}
\td\ro\|v^\al \|^2_{0,\tau} =
 \td\ro \Big( \int_\F v^\al(\tau,x)^2 \,dx  +\int_{\F_\tau^{\al}} (v^\al_x)^2 \, dtdx \Big) \\
\lt 
 \td {\mc A}_{\frac{\ro}4}  \int_{\F_\tau^{\al}}  (1+v^2)  \, dtdx +  \hat{\mc A}\, \al^2 \la_{n+1}(\F^\al_\tau)
 \lt \bar {\mc A} \big(\|v^\al\|^2_{L_2(\F_\tau^\al)} + \al^2 \la_{n+1}(\F^\al_\tau)\big),
}
where $\bar {\mc A} = 3 \td {\mc A}_{\frac{\ro}4} + \hat{\mc A}$.
Further, from inequality (3.7) (p. 76) in \cite{lady} it follows that 
$\|v^\al\|_{L_2(\F_\tau)} \lt \bar\gm  \la_{n+1}(\F^\al_\tau)^{\frac1{n+2}} \|v^\al\|_{0,\tau}$, where $\bar\gm >0$ is a constant
depending on the space dimension $n$. Furthermore, since
$\la_{n+1}(\F^\al_\tau) \lt \tau \la_n(\F)$, we can pick 
$\tau$ sufficiently small such that $\bar {\mc A} \bar\gm^2 (\tau \la_n(\F))^{\frac2{n+2}}\lt \td\ro / 2$.
This implies \rf{boundC}
with $\gm = (2\bar{\mc A} \, \td\ro^{-1})^\frac12$.
\end{proof}
\begin{lem}
\lb{lem98j}
If, under assumptions of Lemma \ref{lem99h}, $f_i = f = 0$, then \rf{boundC} holds for all
$\al \gt \sup_{\ovl \F} |u_0|^2$. 
\end{lem}
\begin{proof}
If $f_i = f = 0$, then the arguments preceding inequality \rf{456h} imply that
\aa{
\td\ro\|v^\al \|^2_{0,\tau}
 \lt \frac12 \int_\F (v^\al(0,x))^2 \,dx   +
 \td {\mc A}_{\frac{\ro}4}  \int_{\F_\tau^{\al}}  v^2  \, dtdx +  \hat{\mc A}\, \al^2 \la_{n+1}(\F^\al_\tau)
}
with $\al \gt \sup_\F |u_0|^2$. The rest of the proof is the same.
\end{proof}
\begin{thm}[Maximum principle for systems of non-local PDEs of form \rf{linearsys}] 
\lb{max-div}
Let assumptions of Lemma \ref{lem99h} be fulfilled. Further let a solution $u$ to problem \rf{linearsys} vanishes on $\pl \F$.
Then $\sup_{\ovl\F_T} |u|$ is bounded by a constant depending only on 
 $\mc A$, $\nu$, $n$, $T$, $\la_n(\F)$, $\hat L_E$, and linearly depending on  $\sup_{\ovl\F} |u_0|$. 
\end{thm} 
\begin{proof}
It follows from Proposition \ref{prop16} and Lemma \ref{lem99h} that there  exists a bound for $\sup_{\ovl\F_\tau}|u|$ 
depending only on $\mc A$, $\nu$, $n$, $\la_n(\F)$, $\hat L_E$, and $\sup_{\ovl \F} |u_0|$, where $\tau\in (0,T]$ is sufficiently small
and depends on  $\mc A$, $\nu$, $n$, $\la_n(\F)$, and $\hat L_E$.
Remark that by Proposition \ref{prop16},  the above bound is a multiple of  $\hat\al = \sup_{\ovl \F} |u_0| +1$.
 It is important to emphasize that $\tau$ 
does not depend on $\sup_{\ovl\F} |u_0|$.
By making the time change $t_1 = t-\tau$ in problem \rf{linearsys}, we obtain a bound for $\sup_{\ovl\F_{\tau,2\tau}}|u|$
depending on  $\mc A$, $\nu$, $n$, $\la_n(\F)$, and $\sup_{\ovl\F} |u(\tau,x)|$,
where the latter quantity was proved to have a bound which is a multiple of $\sup_{\ovl \F} |u_0| +1$.
On the other hand, by Proposition \ref{prop16},
the bound for  $\sup_{\ovl\F_{\tau,2\tau}}|u|$ is a multiple of $\sup_{\ovl\F} |u(\tau,x)|+1$. 
 In a finite number of steps, depending on $T$, we obtain
a bound for $|u|$ in the entire domain $\ovl\F_T$. This bound will depend linearly on 
$\sup_{\ovl\F} |u_0|$ by Proposition \ref{prop16}.
The theorem is proved.
\end{proof}
\begin{cor}
\lb{cor-zero}
Let assumptions of Theorem \ref{max-div} be fulfilled, and let $f=f_i=u_0 = 0$ on $\F_T$. 
Then, $u=0$.
\end{cor}
\begin{proof}
Under assumptions of the corollary, we can set $\al=0$ in Lemma \ref{lem98j}. This implies that
$\|u\|^2_{0,\tau} = 0$ for sufficiently small $\tau\in (0,T]$. Since $u$ is continuous on $\F_T$,
then it is zero on $\ovl \F_\tau$. The same argument as in Theorem \ref{max-div} implies
that $u=0$ on $\ovl \F_T$.
\end{proof}
Since the maximum principle for systems of non-local PDEs of form \rf{linearsys} is obtained, we
can prove the theorem on existence of an a priori bound for $\pl_t u$.
\begin{thm} 
Let (A1)--(A10) hold, and let
$u(t,x)$ be a $\C^{1,2}$-solution to problem \rf{PDE}--\rf{initialsystem}. Then, there exists a constant $M_2$, 
depending only on
$M$, $\hat M$, $M_1$, $\mc K$, $K_{\xi,\zeta}$, $T$, $\la_n(\F)$, $\hat L_E$, $\|\ffi_0\|_{\C^{2+\beta}(\ovl\F)}$, such that
\aa{
\sup_{\ovl \F_T} |\pl_t u| \lt M_2.
}
\end{thm}
\begin{proof}
Rewrite \rf{PDE} in the divergence form, i.e.,
\aaa{
\lb{div-form}
\pl_t u - \sum_{i=1}^n \pl_{x_i} \Big[\sum_{j=1}^n a_{ij}(t,x,u) u_{x_j}\Big] + \hat a(t,x,u,u_x,\tet_u) = 0
\quad \text{with}
}
$\hat a(t,x,u,p,w) =
\sum_{i=1}^n a_i(t,x,u,p,w)p_i + a(t,x,u,p,w)  + \sum_{i,j=1}^n \pl_{x_i} a_{ij}(t,x,u) p_j$
 $+ \sum_{i,j=1}^n (\pl_u a_{ij}(t,x,u),p_i)p_j$, 
where $p_i$ is the $i$th column of the matrix $p$, 
and $u$, $u_x$ and $\tet_u$ are evaluated at $(t,x)$.
 Further, we define
$v(t,x) = (\Dl t)^{-1} \big(u(t+\Dl t,x) - u(t,x)\big)$ and $t' = t+\Dl t$, where $\Dl t$ is fixed.
If $t=0$, we assume that $\Dl t>0$, and if $t=T$, then $\Dl t<0$.
The PDE for the function $v$ takes form \rf{linearsys} with
\eee{
\hat a_{ij}(t,x) = a_{ij} (t',x,u(t',x));\\
A_i(t,x) = \sum_{j=1}^n u_{x_j}(t,x) \int_0^1 d\la \, \pl_u a_{ij}(t,x,\la u(t',x) + (1-\la)u(t,x))^{\! \top};\\
f_i(t,x) =  \sum_{j=1}^n \int_0^1 d\la \, \pl_t a_{ij} (t+\la \Dl t, x, u(t',x)) \,  u_{x_j}(t,x); \\
f(t,x) =    \int_0^1 d\la \, \pl_t \hat a(t+\la \Dl t,x, u(t',x), u_x(t',x), \tet_u(t',x))\\
\hspace{1.1cm} + \int_0^1 d\la\, \pl_w \hat a(t,x,u(t,x),u_x(t,x),\la\tet_u(t',x) + (1-\la)\tet_u(t,x)) 
\big(\xi_{u,u_x}(t,x)\big); \\
A(t,x) =  \int_0^1 d\la \, \pl_u \hat a(t,x,\la u(t',x) + (1-\la) u(t,x), u_x(t',x), \tet_u(t',x))\\
 \hspace{1.1cm} + \int_0^1 d\la\, \pl_w \hat a(t, x,u(t,x),u_x(t,x),\la\tet_u(t',x) + (1-\la)\tet_u(t,x)) \, \zeta_{u,u_x}(t,x); \\
 B_i(t,x) = \int_0^1 d\la \, \pl_{p_i} \hat a(t,x, u(t,x), \la u_x(t',x) + (1-\la) u_x(t,x), \tet_u(t',x));\\
 C(t,x) = \int_0^1 d\la\, \pl_w \hat a(t,x,u(t,x),u_x(t,x),\la\tet_u(t',x) + (1-\la)\tet_u(t,x)).
}
Above, $\xi_{u,u_x}$ and $\zeta_{u,u_x}$ are bounded functions from representation \rf{repa10}.
Remark that the above coefficients are bounded by a constant, say $\mc A$, depending  on $M$, $\hat M$, $M_1$, $\mc K$,
and $K_{\xi,\zeta}$
(where the latter is the bound for $\xi_{u,u_x}$ and $\zeta_{u,u_x}$ defined in (A10)).
By Theorem \ref{max-div}, $\sup_{\ovl\F_T} |v|$ is bounded by a constant depending only on $\mc A$,
$T$, $\la_n(F)$, $\hat L_E$, and $\sup_{\ovl\F }|v(0,x)|$. Moreover, the dependence on 
 $\sup_{\ovl\F }|v(0,x)|$ is linear.
Letting $\Dl t$ go to zero, we obtain that the bound
for $|\pl_t u|$ on $\ovl \F_T$ depends only on $\mc A$,
$T$, $\la_n(F)$, $\hat L_E$, and $\sup_{\ovl\F}|\pl_t u(0,x)|$.
Finally,  equation \rf{PDE} implies that $|\pl_t u(0,x)|$ can be estimated via
$\|\ffi_0\|_{\C^2(\ovl \F)}$, and the bounds for the coefficients $a_{ij}$, $a_i$, and $a$ 
over $\mc R_2$, defined by \rf{r2}. Further, by virtue of (A1) and (A5),  the latter bounds
can be estimated by a constant depending only on $M$, $\hat M$, and $M_1$.
The theorem is proved.
\end{proof}
\subsection{H\"older norm estimates}
In this subsection, we prove that any $\C^{1,2}$-solution to problem \rf{PDE}--\rf{initialsystem}
is, in fact, of class $\C^{1+\frac{\beta}2,2+\beta}$. Moreover, we obtain 
a bound for its $\C^{1+\frac{\beta}2,2+\beta}$-norm.
Unlike the bound for the gradient,  this bound cannot be obtained directly from the results
 of  \cite{lady}  by freezing $\tet_u$. 
Our proof essentially relies  on the estimate of the time derivative $\pl_t u$ obtained in
the previous subsection. 
\begin{thm}(H\"older norm estimate)
\lb{h-n-est-th}
 Let (A1)--(A10) hold, and let 
$u(t,x)$  be a $\C^{1,2}(\ovl \F_T)$-solution to problem  \rf{PDE}--\rf{initialsystem}.
Further let $M$ and $M_1$ be the a priori bounds for $u$ and, respectively, $\pl_x u$ on $\ovl\F_T$ 
(whose existence was established by Theorems  \ref{maxminsys} and \ref{grad-est-th}).
Then, $u(t,x)$ is of class $\C^{1+\frac{\beta}2,2+\beta}(\ovl \F_T)$. Moreover, there exists a constant $M_3>0$
depending only on $M$, $\hat M$, $M_1$, $\mc K$, $K_{\xi,\zeta}$, $T$, $\la_n(\F)$, $\hat L_E$, $\|\ffi_0\|_{\C^{2+\beta}(\ovl\F)}$,   
and on the $\C^{2+\beta}$-norms of the functions defining the boundary $\pl \F$, such that
\aa{
\|u\|_{\C^{1+\frac{\beta}2,2+\beta}(\ovl \F_T)} \lt M_3.
}
\end{thm}
\begin{proof}
Freeze the function $\tet_u$ in the coefficients $a_i$ and $a$, 
and consider the following PDE with respect to $v$
\aaa{
\lb{PDE-lady1} 
- \sum_{i,j=1}^{n} a_{ij} (t,x,v)\pl^2_{x_i x_j}v +  \td a(t,x,v, \pl_x v)  +  \pl_t v=0,
}
where $\td a(t,x,v,p) = a(t,x,v,p,\tet_u(t,x))+ \sum_{i=1}^n  a_i(t,x,v,p,\tet_u(t,x)) p_i$ .
Let us prove that the coefficients of \rf{PDE-lady1} satisfy
the assumptions of Theorem 5.2 from \cite{lady} (p. 587) on the H\"older norm estimate.
First we show that the assumptions on the continuity of the partial derivatives
$\pl_t \td a$, $\pl_u \td a$,  $\pl_p \td a$
and on the $\beta$-H\"older continuity of $\td a$ in $x$, 
mentioned in the formulation of Theorem 5.2 in \cite{lady}, are fulfilled.
Indeed, they follow from (A8) and (A9). To see this, we first note that $a$ 
and $a_i$ depend on $t$ and $x$ not just via their first two arguments but also via 
the function $\tet_u(t,x)$ (assumed known a priori) whose differentiability in $t$ and $x$ follows from (A9).
Therefore, by (A8) and (A9), $a$ and $a_i$ are $\beta$-H\"older continuous in $x$ and 
differentiable in $t$.

Further, Theorem 5.2 of \cite{lady} introduces a common bound  (denote it by $\mc C$) for the partial derivatives
$\pl_t \td a$, $\pl_u \td a$,  $\pl_p \td a$ and the H\"older constant $[\td a]^x_\beta$ 
which, in case of \cite{lady}, exists due to the continuity of the above functions on
$\F_T \x \{|u|\lt M\} \x \{|p|\lt M_1\}$. However, in our case, 
 the expression for $\pl_t \td a$ will contain $\pl_t \tet_u$, and  the expression for $[\td a]^x_\beta$
will contain $\pl_x \tet_u$. Therefore, by (A9), the bound $\mc C$, required for the application of Theorem 5.2,
will depend on $M_1$ and $M_2$, i.e., the bounds for $\pl_x u$ and $\pl_t u$. That is why
the existence of a bound for $\pl_t u$ is indispensable and must be obtained in advance.


The verification of the rest of the assumptions of Theorem 5.2 in \cite{lady} is straightforward
and  follows from assumptions (A1), (A4), (A7), and (A8).
Since $v=u$ is a $\C^{1,2}(\ovl \F_T)$-solution to problem \rf{PDE-lady1}-\rf{initialsystem},
by aforementioned Theorem 5.2, $u$ belongs to class $\C^{1+\frac{\beta}2,2+\beta}(\ovl \F_T)$,
and its H\"older norm $\|u\|_{\C^{1+\frac{\beta}2,2+\beta}(\ovl \F_T)}$
is  bounded by a constant $M_3$,
depending on the constants specified in the formulation of this theorem.
\end{proof}
The rest of this subsection deals with estimates of other H\"older norms of the solution $u$ under assumptions
that do not require the a priori bound $M_2$ for $\pl_t u$. 
These estimates will
be useful in the proof of existence of solution to Cauchy problem \rf{PDE}--\rf{cauchysystem}.
The need of these bounds comes from the fact that $M_2$ depends on $\la_n(\F)$, the Lebesgue
measure of the domain $\F$.
\begin{thm} 
\lb{hn1}
Assume (A1)--(A7).
Let $u(t,x)$ be a generalized $\C^{0,1}(\ovl \F_T)$-solution to equation \rf{PDE} 
such that $|u|\lt M$ and $|\pl_x u| \lt M_1$ on $\ovl \F_T$.
Then, there exist a number $\al\in(0,\beta)$ and a constant 
$M_4$, both depending only on $M$, $M_1$, $\hat M$, $\beta$, $n$, $m$,
and $\sup_{\ovl \F}\|\ffi_0\|_{\C^{2+\beta}(\ovl \F)}$
such that
\aa{
\|u\|_{\C^{\frac\al2, \al}(\ovl \F_T)} \lt M_4.
}
\end{thm} 
\begin{proof}
Freeze the functions $u$, $\pl_x u$, and $\tet_u$ inside the coefficients $a_{ij}$, $a_i$, and $a$, and consider the linear PDE
with respect to $v$
\aaa{
\lb{PDE4}
&\pl_t v - \sum_{i,j=1}^n \td a_{ij}(t,x)\pl^2_{x_ix_j} v + \sum_{i=1}^n \td a_i(t,x)\pl_{x_i}v + \td a(t,x) = 0
\qquad \text{with}\\
\td a_i(t,x) = & a_i(t,x,u,\pl_x u, \tet_u), \quad \td a(t,x) = a(t,x,u,\pl_x u, \tet_u), \quad \td a_{ij}(t,x) = a_{ij}(t,x,u), \notag
}
where $v$, $u$, $\pl_x u$, and $\tet_u$ are evaluated at $(t,x)$. Note that by (A1), (A5), and (A6), 
$a_{ij}$, $\pl_x a_{ij}$, $\pl_u a_{ij}$, $a_i$, and $a$ 
are bounded in the region $\mc R_2$, defined by \rf{r2},
and the common bound depends on $M$, $M_1$, and $\hat M$.
The existence of the bound $M_4$ follows now from Theorem 3.1 of \cite{lady} (p. 582).
\end{proof}
\begin{thm}
\lb{hn2}
Assume (A1)--(A7). Further, assume the following conditions are satisfied in the region
$\mc R_2$, defined by \rf{r2}:
\bi
\item[(i)] $a_{ij}$, $a_i$, $a$ are H\"older continuous in $t,x,u,p$, with exponents 
$\frac\beta2, \beta, \beta, \beta$, respectively, and, moreover, locally Lipschitz 
and  Gat\^eaux differentiable $w$; all H\"older and Lipschitz constants are bounded (say, by a constant $\mc M$);
\item[(ii)] For any $\C^{1,2}(\ovl \F_T)$-solution $u(t,x)$ to problem  \rf{PDE}--\rf{initialsystem} and for some 
$\beta'\in (0,\beta)$, 
the bound for $[\tet_u]^t_{\frac{\beta'}2}$ is determined 
only by the bound for $[u]^t_{\frac{\beta'}2}$ and $M_1$; 
and the bound for $[\tet_u]^x_{\beta'}$ is determined only by $M_1$.
\ei
Let $u(t,x)$  be a $\C^{1,2}(\ovl \F_T)$-solution to equation  \rf{PDE}
such that  $|u|\lt M$ and $|\pl_x u| \lt M_1$ on $\ovl \F_T$,
and let $\G \sub \F$ be a strictly interior open domain.
Then,  there exist a number $\al\in (0,\beta\we \beta')$ and a constant $M_5$, both depending
 only on $M$, $M_1$, $\hat M$,  $\mc M$,
$\|\ffi_0\|_{\C^{2+\beta}(\ovl\F)}$, 
and the distance between $\ovl \G$ and $\pl \F$, such that $u$
is of class $\C^{1+\frac{\al}2, 2+\al}(\ovl\G_T)$, and
\aa{
\|u\|_{\C^{1+\frac{\al}2, 2+\al}(\ovl \G_T)} \lt M_5.
}
\end{thm}
\begin{proof}
Freeze the function $\tet_u$ in the coefficients $a_i$, and $a$, and consider 
PDE \rf{PDE-lady1}  with respect to $v$.
 Let $\al$ be the smallest of $\beta'$ and the exponent
 whose existence was established by Theorem \ref{hn1}.
Assumptions (i) and (ii) imply that the coefficient $\td a$ in PDE \rf{PDE-lady1} is H\"older continuous
in $t,x,u$, and $p$ with exponents $\frac\al2, \al, \al$, and $\al$, respectively. 
Moreover, the H\"older constants
are bounded and their common bound depends on $\mc M$, $M_1$, and $M_4$.
The constant $M_4$, in turn, depends on 
$M$, $M_1$, $\hat M$, $\beta$, and $\sup_{\ovl \F}\|\ffi_0\|_{\C^{2+\beta}(\ovl \F)}$.
Thus, by Theorem 5.1 of \cite{lady} (p. 586),  the solution $u$ is of class 
$\C^{1+\frac{\al}2, 2+\al}(\ovl\G_T)$ and the bound for the norm 
$\|u\|_{\C^{1+\frac{\al}2, 2+\al}([0,T]\x \ovl \G)}$
depends only on $M$, $M_1$, $\hat M$, $\mc M$, $\sup_{\ovl \F}\|\ffi_0\|_{\C^{2+\al}(\ovl \F)}$, and the distance 
between $\ovl \G$ and $(\pl \F)_T$.
The theorem is proved.
\end{proof}

\subsection{Existence and uniqueness for the initial--boundary value problem}
To obtain the existence and uniqueness result for problem \rf{PDE}--\rf{initialsystem}, 
we need the two additional assumptions below:
\bi
\item[\bf (A11)] The following compatibility condition 
  holds for $x\in \pl \F$:
\mm{ 
- \sum_{i,j=1}^n  a_{ij} (0,x,0)\pl^2_{x_i x_j} \ffi_0(x) +\sum_{i=1}^n   a_i (0,x,0,\pl_x \ffi_0(x), \tet_{\ffi_0}(0,x)) \pl_{x_i} \ffi_0(x)\\
 + a(0,x, 0, \pl_x \ffi_0(x), \tet_{\ffi_0}(0,x)) =0. 
  }
\item[\bf (A12)] For any $u, u' \in \C^{1,2}_0(\ovl \F_T)$, it holds that 
\aaa{
 \hspace{1.2cm}   
 \lb{tdtet}
 \tet_{u}(t,x) - \tet_{u'}(t,x)  = \td \tet_{u-u'}(t,x) +  \sig_{u,u',u_x, u'_x}(t,x) (u(t,x)-u'(t,x)),
}
 where $\sig_{u,u',u_x, u'_x}: \ovl \F_T \to \mc L(\Rnu^m,E)$ is bounded and may depend
non-locally on $u$, $u'$, $u_x$, and $u'_x$;
$\td \tet_v: \ovl \F_T \to E$ is defined for each $v\in \C^{1,2}_0(\ovl \F_T)$
and satisfies (A2) (in the place of $\tet_u$).
\ei
The main tool in the proof of  existence for initial--boundary value problem
\rf{PDE}-\rf{initialsystem}
is the following version of the Leray-Schauder theorem proved in \cite{ls} (Theorem 11.6, p. 286).
First, we recall that a map is called \textit{completely continuous} if it takes bounded sets into relatively compact sets. 

 \begin{thm}(Leray-Schauder theorem)
 \lb{leray}
 \label{lerayschauder}
Let $X$ be a Banach space, and let $\Phi$ be a completely continuous map
$[0,1]  \times X \rightarrow X$ such that for all $x\in X$, $\Phi(0,x)= c \in X$. Assume there exists a constant $K>0$ such that
 for all $(\tau, x) \in [0,1]  \times X $ solving the equation
  $\Phi(\tau,x) = x$, it holds that $\|x\|_X < K$. Then, the map $\Phi_1(x) = \Phi(1,x)$
has  a fixed point.
\end{thm}
\begin{rem}
\rm
Theorem 11.6 in \cite{ls} is, in fact, proved for the case $c=0$. However, let us observe that the assumptions of Theorem 11.6
are fulfilled for the map $\td \Phi(\tau,x) = \Phi(\tau,x+c) - c$, whenever $\Phi$ satisfies the assumptions of Theorem \ref{lerayschauder}.
To see this, we first check that $\td \Phi$ is completely continuous.
Let $B\sub [0,1] \x X$ be a bounded set, then $B' = \{(\tau,x+c)\; \text{s.t.} \; (\tau,x) \in B\}$ is also a bounded 
set with the property
$\td \Phi(B) = \Phi(B') - c$. Therefore, $\td \Phi$ is completely continuous if and only if $\Phi$ is completely continuous. 
Next, it holds that $\td \Phi(0,x) = 0$ for all $x\in X$. It remains to note that
$x$ is a fixed point of the map $\td \Phi(\tau,\fdot)$ if and only if $x+c$ is a fixed point of the map $\Phi(\tau,\fdot)$.
\end{rem}
Now we are ready to prove the main result of Section \ref{s2} which is the
existence and uniqueness theorem for  non-local initial--boundary value problem
\rf{PDE}-\rf{initialsystem}.
\begin{thm}[Existence and uniqueness for initial--boundary value problem]
\lb{existence}
Let  (A1)--(A11) hold.   
Then,  there exists a $\C^{1+\frac\beta2,2+\beta}(\ovl \F_T)$-solution 
 to non-local initial--boundary value problem \rf{PDE}-\rf{initialsystem}. If, in addition, (A12) holds, then
this solution is unique.
\end{thm}
\begin{proof}
\textit{Existence.}
For each $\tau\in [0,1]$, consider the initial--boundary value problem 
\eq{
\pl_t u - \sum_{i,j=1}^n (\tau a_{ij} (t,x,u)  + (1-\tau) \dl_{ij})\pl^2_{x_i x_j} u + (1-\tau)\Dl \ffi_0 \\
+ \tau  \sum_{i=1}^n a_i(t,x,u, \pl_x u, \tet_u)\pl_{x_i} u + \tau \, a(t,x,u, \pl_x u, \tet_u)=0,\\
u(0,x) =\ffi_0(x), \quad u(t,x)\big|_{(\pl \F)_T} = 0,
 \label{cont-prob}
}
where $u$, $u_x$, and $\tet_u$ are evaluated at $(t,x)$.
In the above equation, we freeze $u\in \C^{1,2}(\ovl \F_T)$ whenever it is in the arguments of the coefficients $a_{ij} (t,x,u)$, $a_i(t,x,u, \pl_x u, \tet_u)$,
$a(t,x,u, \pl_x u, \tet_u)$, and consider the following linear initial--boundary value problem with respect to $v$:
\eq{
\pl_t v^k - \sum_{i,j=1}^n \big( \tau a_{ij} (t,x,u)  + (1-\tau) \dl_{ij}\big)\pl^2_{x_i x_j} v^k + (1-\tau)\Delta \ffi^k_0 \\
+ \tau  \sum_{i=1}^n a_i(t,x,u, \pl_x u, \tet_u)\pl_{x_i} v^k + \tau \, a^k(t,x,u, \pl_x u, \tet_u)=0,\\
v^k(0,x) =\ffi^k_0(x), \quad v^k(t,x)\big|_{(\pl \F)_T} = 0,
 \label{linearprob}
}
where $v^k$, $\ffi^k_0$,  and $a^k$ are the $k$th components of $v$, $\ffi_0$, and $a$, respectively. 
Remark that the assumptions of Theorem 5.2, Chapter IV in \cite{lady} (p. 320)
on the existence and uniqueness of solution for linear parabolic PDEs are fulfilled for
problem \rf{linearprob}. Indeed, the assumptions of Theorem 5.2 in \cite{lady} require that the coefficients
of \rf{linearprob} are of class $\C^{\frac{\beta}2,\beta}(\ovl \F_T)$ for some $\beta\in (0,1)$.
This holds by (A8), (A9), and (A4).
The assumption about the boundary $\pl \F$ and the boundary function $\psi$ is fulfilled by (A4) and (A7).
Finally, the compatibility condition on the boundary $\pl \F$, required by Theorem 5.2, follows from (A11). 
Therefore, by Theorem 5.2 (p. 320) in \cite{lady}, we conclude that there exists a unique solution $v^k(t,x)$ to problem \rf{linearprob} 
which belongs to class $\C^{1+\frac{\beta}2, 2+ \beta}(\ovl \F_T)$. 
Clearly, $v^k$  is also of class  $\C^{1,2}(\ovl \F_T)$, and, therefore, for each $\tau \in [0,1]$,
we have the map $\Phi: \C^{1,2}(\ovl \F_T) \to \C^{1,2}(\ovl \F_T)$, 
$\Phi(\tau,u) = v$. 
Note that, fixed points of the map $\Phi(\tau,\fdot)$, if any, are solutions of \rf{cont-prob}. In particular,
fixed points of $\Phi(1,\fdot)$ are solutions to original problem \rf{PDE}-\rf{initialsystem}.

To prove the existence of fixed points of the map $\Phi(1,\fdot)$, we apply the Leray--Schauder theorem (Theorem \ref{leray}).
Let us verify  its conditions. First we note that if $\tau=0$, then the PDE in  \rf{linearprob} takes the form 
$\pl _t v^k - \lap v^k + \lap \ffi_0^k = 0$. 
Therefore, it holds that $\Phi(0, u) = \ffi_0$ for all $u\in \C^{1,2}(\ovl \F_T)$.
Let us prove that $\Phi$ is completely continuous. 
 Suppose $B \sub [0,1] \x \C^{1,2}(\ovl \F_T)$ is a bounded set, i.e., for all $(\tau, u)\in B$, it holds that $\|u\|_{\C^{1,2}(\ovl \F_T)}\lt \gm_B$ for some constant $\gm_B$ depending on $B$.
 By aforementioned Theorem 5.2 from \cite{lady} (p. 320), the solution $v_{\tau,u}(t,x) = \{v_{\tau,u}^k(t,x)\}_{k=1}^m$ 
 to problem \rf{linearprob}, corresponding to the pair $(\tau,u)\in B$,
satisfies the estimate
\aa{
\|v_{\tau,u}\|_{\C^{1+\frac{\beta}2, 2+ \beta}(\ovl \F_T)} \lt &\gm_1 \big(\|a(t,x,u(t,x),\pl_x u(t,x), \tet_u(t,x))\|_{\C^{\frac{\beta}2, \beta}(\ovl \F_T)}
+ \|\ffi_0\|_{\C^{2+\beta}(\ovl \F_T)}\big),
}
where the first term on the right-hand side is bounded by (A8), (A9), and by the boundedness of 
 $\|u\|_{\C^{1,2}(\ovl \F_T)}$ for all $(\tau,u)\in B$. 
 Moreover, the bound for this term depends only on $\gm_B$ and $\mc K$
(where $\mc K$ is the constant defined in Remark \ref{rem8}).
This implies that 
$\|v_{\tau, u}\|_{\C^{1+\frac{\beta}2, 2+ \beta}(\ovl \F_T)}$ is bounded by a constant that depends only on $\mc K$, $\gm_B$, $\gm_1$, and
$\|\ffi_0\|_{\C^{2+\beta}(\ovl \F_T)}$.  By the definition of the norm in $\C^{1+\frac{\beta}2, 2+ \beta}(\ovl \F_T)$ (see 
\rf{parabolic-holder-norm}), the family $v_{\tau,u}$, $(\tau,u)\in B$, is uniformly bounded and
uniformly continuous in $\C^{1,2}(\ovl \F_T)$.
By the Arzel\'a--Ascoli theorem, $\Phi(B)$ is relatively compact, and, therefore,
the map $\Phi$ is completely continuous.

It remains to prove that there exists a constant $K>0$ such that for each $\tau \in [0,1]$ and for each
$\C^{1,2}(\ovl \F_T)$-solution $u_\tau$ to problem \rf{cont-prob}, it holds that $\|u_\tau\|_{\C^{1,2}(\ovl \F_T)}\lt K$.
Remark that the coefficients of problem \rf{cont-prob} satisfy (A1)--(A10). Hence, 
by Theorem \ref{h-n-est-th}, the H\"older norm $\|u_\tau\|_{\C^{1+\frac{\beta}2,2+\beta}(\ovl \F_T)}$,
and, therefore, the $\C^{1,2}(\ovl \F_T)$-norm of $u_\tau$,
is bounded by a constant depending only on 
$M$, $\hat M$, $M_1$, $\mc K$, $K_{\xi,\zeta}$, $T$, $\la_n(\F)$, $\hat L_E$, $\|\ffi_0\|_{\C^{2+\beta}(\ovl\F)}$,   
and on the $\C^{2+\beta}$-norms of the functions defining the boundary $\pl \F$.

Thus, the conditions of Theorem \ref{leray} are fulfilled. This implies the existence of a fixed point of the map $\Phi(1,\fdot)$,
and, hence, the existence of a $\C^{1,2}(\ovl \F_T)$-solution to problem \rf{PDE}-\rf{initialsystem}.
Further, by Theorem \ref{h-n-est-th}, any $\C^{1,2}(\ovl \F_T)$-solution to problem \rf{PDE}-\rf{initialsystem}
is of class $\C^{1+\frac\beta2,2+\beta}(\ovl \F_T)$.

\textit{Uniqueness.} Let us prove the uniqueness under (A12).  Rewrite \rf{PDE} in divergence form
\rf{div-form}.
Suppose now  $u$ and $u'$ are two solutions to \rf{PDE}-\rf{initialsystem} of class $\C^{1,2}(\ovl \F_T)$. 
Define $v =u-u'$.  The PDE for the function $v$ takes form \rf{linearsys}
with
\eee{
\hat a_{ij}(t,x) = a_{ij} (t,x,u'(t,x));\\
A_i(t,x) = \sum_{j=1}^n u_{x_j}(t,x) \int_0^1 d\la \, \pl_u a_{ij}(t,x,\la u'(t,x) + (1-\la)u(t,x))^{\! \top};\\
A(t,x) =  \int_0^1 d\la \, \pl_u \hat a(t,x,\la u'(t,x) + (1-\la) u(t,x), u'_x(t,x), \tet_{u'}(t,x))\\
+ \int_0^1 d\la\, \pl_w \hat a(t, x,u(t,x),u_x(t,x),\la\tet_{u'}(t,x) + (1-\la)\tet_u(t,x)) \, \sig_{u,u',u_x, u'_x}(t,x); \\
 B_i(t,x) = \int_0^1 d\la \, \pl_{p_i} \hat a(t,x, u(t,x), \la u_x(t,x) + (1-\la) u_x(t,x), \tet_{u'}(t,x));\\
 C(t,x) = \int_0^1 d\la\, \pl_w \hat a(t,x,u(t,x),u_x(t,x),\la\tet_{u'}(t,x) + (1-\la)\tet_u(t,x)); \\
 f_i(t,x) = f(t,x) = u_0(x) = 0.
}
Above, $\sig_{u,u',u_x, u'_x}$
is a bounded function from representation \rf{tdtet}.
Remark that the above coefficients are bounded by a constant, say $\mc A$, depending  on $M$, $\hat M$, $M_1$,
$\mu(M,\hat M)$, and $\mc K$. By Corollary \ref{cor-zero}, $v=0$ on $\ovl\F_T$.
\end{proof}

\subsection{Existence and uniqueness for the Cauchy problem}
\label{s23}
In this subsection, we consider Cauchy problem \rf{PDE}--\rf{cauchysystem}. 
The results of the previous subsection 
will be used to prove the existence theorem for this problem.

Below, we formulate assumptions (A1')--(A12') needed for
the existence and uniqueness theorem. 
Assumptions (A1')--(A3') are the same as (A1)--(A3) but 
 $\F$ should be replaced with $\Rnu^n$, and
$\C^{1,2}_0(\ovl \F_T)$ with $\C^{1,2}_b([0,T]\x \Rnu^n)$. Also, $\tet_u$ is
defined for all $u\in \C^{1,2}_b([0,T]\x \Rnu^n)$.

As before, the functions $\mu(s)$, $\hat \mu(s)$, $\td \mu(s)$, $\eta(s,r)$, $P(s,r,t)$, $\eps(s,r)$ 
are continuous,  defined for positive arguments,
taking positive values, and non-decreasing (except $\hat\mu(s)$) with respect to each argument, 
whenever the other arguments are fixed; the
function $\hat \mu(s)$ is non-increasing. Further, $\mc {\td R}$, $\mc {\td R}_1$, $\mc {\td R}_2$, and $\mc {\td R}_3$ are 
defined as follows
\aa{
&\mc {\td R}  = [0,T]\x \Rnu^n \x \Rnu^m \x  \RR^{m \x n} \x E; \quad \mc {\td R}_1 = [0,T]\x \Rnu^n\x\Rnu^m;\\
&\mc {\td R}_2 = [0,T] \x \Rnu^n \x \{|u|\lt C_1\} \x \{|p|\lt C_2\} \x \{\|w\|_E \lt C_3\}; \\
& \mc {\td R}_3 = [0,T] \x \{|x|\lt C_1\} \x \{|u|\lt C_2\} \x \{|p|\lt C_3\} \x \{\|w\|_E \lt C_4\},
}
where $C_1$, $C_2$, $C_3$, $C_4>0$ are arbitrary constants.  
Assumptions (A4')--(A12') read:
\bi
\item[\textbf{(A4}\textrm{'}\textbf{)}] The initial condition $\ffi_0:\Rnu^n \to \Rnu^m$ is of class $\C_b^{2+\beta}(\Rnu^n)$, $\beta\in (0,1)$.
\item[\textbf{(A5}\textrm{'}\textbf{)}]  For all $(s,x,u,p,w)\in \mc {\td R} $, 
   \aa{
   & |a_i (t,x,u,p,w)| \leq \eta(|u|, \|w\|_E)(1+|p|), \\
   & |a(s,x,u,p,w)| \lt \big(\eps(|u|, \|w\|_E) + P(|u|, \|w\|_E, |p|)\big)(1+|p|)^2,
  }
 where $\lim_{r\to \infty}P(s,r,q) = 0$ and $2(s+1) \eps(s,r) \lt \hat\mu(s)$.
 \item[\textbf{(A6}\textrm{'}\textbf{)}] 
 $\pl_x a_{ij}$, $\pl_u a_{ij}$, $\pl_t a_{ij}$, $\pl^2_{uu} a_{ij}$, 
 $\pl^2_{ux} a_{ij}$, $\pl^2_{xt} a_{ij}$, and $\pl^2_{ut} a_{ij}$ exist and  
 are continuous on $\mc {\td R}_1$; moreover, 
$\max \big\{\big| \pl_x a_{ij}(t,x,u)\big| , \big|\pl_u a_{ij}(t,x,u)\big| \big\}\lt \td \mu(|u|)$.
 \item[\textbf{(A7}\textrm{'}\textbf{)}] 
 $\pl_t a$, $\pl_u a$, $\pl_p a$, $\pl_w a$, $\pl_t a_i$, $\pl_u a_i$, $\pl_p a_i$, and $\pl_w a_i$
exist and are  bounded and continuous on regions of form $\mc {\td R}_2$; $a$ and $a_i$ are  
$\beta$-H\"older continuous in $x$ and locally Lipschitz in 
$w$ with the H\"older and Lipschitz constants  bounded in regions of form $\mc {\td R}_2$.
 \item[\textbf{(A8}\textrm{'}\textbf{)}] The same as (A9) but valid for any bounded domain  $\F$.
 \item[\textbf{(A9}\textrm{'}\textbf{)}] The same as (A10) but valid for any bounded domain  $\F$.
 \item[\textbf{(A10}\textrm{'}\textbf{)}]
For any bounded domain $\F\sub \Rnu^n$, for some $\al\in (0,\beta)$,
the bound for $[\tet_u]^t_{\frac{\al}2}$ on $\ovl\F_T$  
is determined only by the bounds for $[u]^t_\al$ and $\pl_x u$, and the bound for 
$[\tet_u]^x_\al$ is determined only
by the bound for $\pl_x u$.

\item[\textbf{(A11}\textrm{'}\textbf{)}]
For all $u, u' \in \C^{1,2}_b([0,T]\x \Rnu^n)$,  representation \rf{tdtet} holds with
$\td \tet_v: [0,T]\x \Rnu^n \to E$, defined for each $v\in \C^{1,2}_b(\Rnu^n\x [0,T])$ and 
satisfying the inequality
$\sup_{[0,t]\x \Rnu^n}\|\td\tet_v\|_E \lt L_E \sup_{[0,t]\x \Rnu^n} |v|$ for all $t\in (0,T]$;
$\sig_{u,u',u_x, u'_x}(t,x)$ in \rf{tdtet} are bounded, continuous, 
 and $\beta$-H\"older continuous in $x$.
\item[\textbf{(A12}\textrm{'}\textbf{)}] 
$\pl^2_{xx} a_{ij}$, $\pl^2_{xu} a_{ij}$, $\pl^2_{uu} a_{ij}$, $\pl_t a_{ij}$,
 $\pl_x a_i$,  $\pl_u a_i$,  $\pl_p a_i$, $\pl_w a_i$, $\pl_u a$, $\pl_p a$, $\pl_w a$,  
$\pl^2_{p x} a$,  $\pl^2_{p u} a$, $\pl^2_{p p} a$,  $\pl^2_{p w} a$, $\pl^2_{p x} a_i$, $\pl^2_{p u} a_i$,   
$\pl^2_{p p} a_i$, $\pl^2_{p w} a_i$  
exist and are 
 bounded and continuous on regions of form $\mc {\td R}_2$, and, moreover, $\al$-H\"older continuous 
 in $x$, $u$, $p$, $w$ for some $\al\in (0,1)$; $\pl_p a$ and $\pl_p a_i$ are locally Lipschitz in $w$.
Furthermore, all the Lipschitz  constants are bounded over regions of form $\mc {\td R}_2$, 
 and all the H\"older  constants are bounded over regions of form $\mc {\td R}_3$.
\ei
Assumptions (A11')--(A12') are required only for the proof of uniqueness. Unlike initial--boundary value problems, 
we do not prove a maximum principle for Cauchy problems. 
 The uniqueness result for problem  \rf{PDE}-\rf{cauchysystem} follows 
 from the possibility to solve linear parabolic systems via fundamental solutions.

 Theorem \ref{existence-cauchy} below is one of our main results.
\begin{thm}[Existence and uniqueness for the Cauchy problem]
\lb{existence-cauchy}
Let  (A1')--(A10') hold.   
Then,  there exists a $\C^{1,2}_b([0,T]\x\Rnu^n)$-solution 
 to non-local Cauchy problem \rf{PDE}-\rf{cauchysystem} which, moreover, belongs to class
 $\C^{1+\frac{\al}2,2+\al}_b([0,T]\x\Rnu^n)$
 for some $\al \in (0,\beta)$.
 If, additionally, (A11') and (A12') hold, then this solution is unique.
\end{thm}
\begin{proof}
\textit{Existence.}
We employ the diagonalization argument similar to the one presented in \cite{lady} (p. 493) 
for the case of one equation. 
Consider PDE \rf{PDE} on the ball $B_r$ of radius $r>1$  with the boundary function 
\aaa{
\lb{bf2}
\psi(t,x) =
\begin{cases} 
\ffi_0(x) \xi(x), \quad x\in \{t=0\} \times B_r, \\
0, \quad (t,x) \in  [0,T] \x \pl B_r,
\end{cases}
}
where $\xi(x)$ is a smooth function such that $\xi(x) =1$ if $x\in B_{r-1}$,
$\xi(x) = 0$ if $x\notin B_r$; further, $\xi(x)$ decays from $1$ to $0$ along the radius on $B_r\dd B_{r-1}$
in a way that $\xi^{(l)}(x)$, $l=1,2,3$, does not depend on $r$ and are zero on $\pl B_r$.
Let $u_r(t,x)$ be the 
$\C^{1+\frac\beta2,2+\beta}(\ovl B_{r+1})$-solution to problem \rf{PDE}-\rf{bf2} in the ball $B_{r+1}$ 
whose existence was established by Theorem \ref{existence}.
Remark, that since $u_r$ is zero on $\pl B_{r+1}$, it can be extended by zero to the entire space $\Rnu^n$,
and, therefore, $\tet_{u_r}$ is well-defined.
Moreover, by Theorem \ref{max-princ},
on $B_{r+1}$  the solution $u_r$ is bounded by a constant $M$ that only depends on
$T$, $L_E$, $\sup_{\Rnu^n} |\ffi_0|$, and the constants $c_1$, $c_2$, $c_3$ from (A3'). 
Next, by Theorem \ref{grad-est-th}, the  gradient $\pl_x u_r$ possesses a bound $M_1$ on $B_{r+1}$
which only depends on $M$, $\hat M$, $\sup_{\Rnu^n}|\pl_x \ffi_0|$, 
$\mu(M)$,  $\hat\mu(M)$, $\td\mu(M)$, $\eta(M,\hat M)$, $\sup_{q\gt 0} P(M,q,\hat M)$, and $\eps(M,\hat M)$.
Thus, both bounds $M$ and $M_1$ do not depend on $r$.

Remark that the partial derivatives and H\"older constants mentioned in assumption (A7')
are bounded in the region $[0,T]\x \Rnu^n \x \{|u|\lt M\} \x \{|p|\lt M_1\} \x \{\|w\|_E \lt \hat M$\}. 
Let $\mc K$
be their common bound. 

Fix a ball $B_R$ for some $R$. 
By Theorem \ref{hn2}, there exist $\al\in (0,\beta)$ and a constant $C>0$, both depend 
only on $M$, $M_1$, $\hat M$, $\mc K$, and $\|\ffi_0\|_{\C^{2+\beta}(\Rnu^n)}$, such that 
$\|u_r\|_{\C^{1+\frac{\al}2,2+\al}([0,T]\x \ovl B_r)} \lt C$ 
(remark  that the distance between $\ovl B_r$ and $\pl B_{r+1}$ equals one). Therefore,
$\|u_r\|_{\C^{1+\frac{\al}2,2+\al}([0,T]\x \ovl B_R)} \lt C$ for all $r>R$.
It is important to mention that the constant $C$ does not depend on $r$. By the Arzel\`a-Ascoli theorem, 
the family of functions $u_r(t,x)$, parametrized by $r$, 
is relatively compact in $\C^{1,2}([0,T]\x \ovl B_R)$.
Hence, the family $\{u_r\}$ contains a sequence
$\{u^{(0)}_{r_k}\}_{k=1}^\infty$ which converges in $\C^{1,2}([0,T]\x \ovl B_R)$.
Further, we can choose a subsequence $\{u^{(1)}_{r_k}\}_{k=1}^\infty$ of  $\{u^{(0)}_{r_k}\}_{k=1}^\infty$ 
with $r_k>R+1$ that converges in  $\C^{1,2}([0,T]\x \ovl B_{R+1})$.
Proceeding this way we find a subsequence $\{u^{(l)}_{r_k}\}$ with $r_k>R+l$ that converges in  
$\C^{1,2}([0,T]\x \ovl B_{R+l})$.
The diagonal sequence $\{u^{(k)}_{r_k}\}_{k=1}^\infty$ converges pointwise on 
$[0,T]\x \Rnu^n$ to a function $u(t,x)$,
while its derivatives $\pl_t u^{(k)}_{r_k}$, $\pl_x u^{(k)}_{r_k}$, and $\pl^2_{xx} u^{(k)}_{r_k}$
converge pointwise on $[0,T]\x \Rnu^n$ to the corresponding derivatives of $u(t,x)$. Therefore,
$u(t,x)$ is a $\C^{1,2}_b$-solution to problem \rf{PDE}-\rf{cauchysystem}.

Let us prove that $u\in \C_b^{1+\frac{\al}2,2+\al}([0,T]\x\Rnu^n)$. Note that $|u|\lt M$
and $|\pl_x u|\lt M_1$, where $M$ and $M_1$ are bounds for $|u_r|$ and, respectively $|\pl_x u_r|$,
that are independent of $r$. By Theorem \ref{hn2}, 
$\|u\|_{\C^{1+\frac{\al}2,2+\al}([0,T]\x \ovl B_R)} \lt C$, where the constants $\al$
and $C$ are the same as for $u_r$. Moreover, the above estimate
holds for any ball $B_R$.
Therefore, $\|u\|_{\C^{1+\frac{\al}2,2+\al}_b([0,T]\x \Rnu^n)} \lt C$.

\textit{Uniqueness.} 
Rewrite \rf{PDE} in the form
\aaa{
\lb{PDE-lady2} 
- \sum_{i,j=1}^{n} a_{ij} (t,x,u)\pl^2_{x_i x_j}u +  \td a(t,x,u, \pl_x u, \tet_u)  +  \pl_t u=0,
}
where $\td a(t,x,u,p, w) = a(t,x,u,p,w)+ \sum_{i=1}^n  a_i(t,x,u,p,w) p_i$ with $p_i$ being the $i$th column of the matrix $p$.
As before, $u$, $\pl_x u$, $\pl_t u$, and $\tet_u$ are evaluated at $(t,x)$. 

Suppose we have two $\C^{1,2}_b$-solutions $u$ and $u'$ to Cauchy problem
\rf{PDE-lady2}--\rf{cauchysystem}. Then $v=u-u'$ is a solution to 
\aa{
\begin{cases}
 \pl_t u - \suml_{i,j=1}^{n} \td a_{ij} (t,x)\pl^2_{x_i x_j}u + \suml_{i=1}^n B_i (t,x)\pl_{x_i}u
+ A(t,x)u + C(t,x)\big(\td \tet_u\big) = f(t,x), \\
u(0,x) = 0, \quad [0,T]\x \Rnu^n.
\end{cases}
}
with the coefficients 
\eee{
\td a_{ij}(t,x) = a_{ij} (t,x,u(t,x)),\\
A(t,x)  = - \sum_{i,j=1}^n \pl^2_{x_i x_j} u'(t,x) \int_0^1 d\la \, \pl_u a_{ij}(t,x,\la u'(t,x) + (1-\la)u(t,x))^\top   \\
\hspace{3mm} + \int_0^1 d\la\, \pl_u \td a(t,x, \la u'(t,x) + (1-\la)u(t,x), \pl_x u(t,x), \tet_u(t,x)) \\
\hspace{3mm} + \int_0^1 d\la  \, \pl_w \td a(t,x,u'(t,x), \pl_x u'(t,x), \la \tet_{u'}(t,x) + (1-\la) \tet_u(t,x))\, \sig_{u,u',u_x, u'_x}(t,x),\\
B_i(t,x) = \int_0^1 d\la\, \pl_{p_i} \td a(t,x,u'(t,x), \la\pl_x u'(t,x) + (1-\la)\pl_x u(t,x),  \tet_u(t,x)),\\
C(t,x) = \int_0^1 d\la  \, \pl_w \td a(t,x,u'(t,x), \pl_x u'(t,x), \la \tet_{u'}(t,x) + (1-\la) \tet_u(t,x)),\\
f(t,x) = 0,
}
where $\sig_{u,u',u_x, u'_x}$ is defined by decomposition \rf{tdtet}.
Assumptions (A1'), (A6'), (A7'), and (A10')--(A12') 
imply the conditions of Theorems 3 and 6 in \cite{friedman} (Chapter 9, pp.  256 and 260)
on the existence and uniqueness of solution to a system of linear parabolic PDEs via the fundamental
solution  $G(t,x;\tau, z)$. Namely, the forementioned Theorems 3 and 6  imply that
the function $v$ satisfies the equation
\aa{
v(t,x) = \int_0^t \int_{\Rnu^n} G(t,x;\tau, z)C(\tau,z)\big(\td\tet_v(\tau,z)\big)d\tau dz.
} Further, (A11') and (A12') imply the boundedness of $C(t,z)$ and $\td\tet_v(\tau,z)$. 
Finally, taking into account the
estimate  $\sup_{[0,t]\x \Rnu^n}\|\td \tet_v\|_E \lt L_E\sup_{[0,t]\x \Rnu^n}\sup |v|$, 
as well as Theorem 2 in \cite{friedman}
(Chapter 9, p. 251) which provides an estimate for the fundamental solution via a Gaussian-density-type function, 
by Gronwall's inequality, we  obtain that $v(t,x) = 0$.
Therefore, a $\C^{1,2}_b$-solution to \rf{PDE}-\rf{cauchysystem} is unique.
\end{proof}

\section{Fully-coupled FBSDEs with jumps}	
\label{s3}

In this section,  we
obtain an existence and uniqueness theorem for  FBSDEs with jumps by means of the 
results of Section \ref{s2}. 

Let $(\Omega, \mc F, \mc F_t, \PP)$  be a filtered probability space with
the augmented filtration $\mc F_t$ satisfying the usual conditions. 
Assume that the filtration $\mc F_t$ is generated by the following two 
mutually independent processes:
 an $n$-dimensional standard Brownian motion $B_t$ and a
 Poisson random measure $N(t,\fdot)$  on $\Rnu_+ \x \mf B(\Rnu^l_*)$,
where $\Rnu^l_* = \Rnu^l-\{0\}$ and $\mf B(\Rnu^l_*)$ is the $\sg$-algebra of Borel sets.
Further let
$\td N(t,A)=N(t,A) - t\nu(A)$ be the associated
compensated Poisson random measure on $\Rnu_+ \x \mf B(\Rnu^l_*)$, and $\nu(A)$ be its intensity which
is assumed to be a L\'evy measure.

Fix an arbitrary $T>0$ and consider FBSDE \rf{eqFBSDE}.
  { By a \textit{solution} to \rf{eqFBSDE} we understand  an $\mc F_t$-adapted quadruple 
 $(X_t,Y_t,Z_t, \td Z_t)$ with values in $\Rnu^n\x \Rnu^m \x \Rnu^{m\x n} \x L_2(\nu, \Rnu^l_* \to \Rnu^m)$
such that \rf{eqFBSDE} is fulfilled a.s. 
The latter includes the existence of all stochastic integrals involved in \rf{eqFBSDE}. In particular,
$\td Z_t$ is $\mc F_t$-predictable and such that  $\E\int_0^T \int_{\Rnu^l_*} |\td Z(t,y)|^2 \nu(dy) dt <+ \infty$ and,
furthermore,  $\E\int_0^T  |Z_t|^2 dt <+ \infty$. Also remark that we implicitly assume the existence
of the left limits for the pair $(X_t,Y_t)$, and, in fact, we will be interested in  c\`adl\`ag versions of $(X_t,Y_t)$,
so the aforementioned requirement would be automatically fulfilled.}

Together with FBSDE \rf{eqFBSDE}, we
consider the associated final value problem for the following partial integro-differential equation:
\eq{
\lb{eqPIDEFBSDE1}
\pl_x\te \big\{ f\big(t,x,{\theta},\pl_x\te\, \sg(t,x, \te), \tet_\te(t,x)\big)  
  -  \int_{\Rnu^l_*} {\ffi} (t,x,{\theta},y)\nu(dy)\big\} \\
   + \frac12 \tr \big(\pl^2_{xx}\te \,\sg(t,x,\te) \sg(t,x,\te)^\top\big)
+  g\big(t,x,{\theta},\pl_x\te\sg(t,x, \te), \tet_\te(t,x)  \big)  \\
+  \int_{\Rnu^l_*}  \tet_\te(t,x)(y) \nu(dy)  +\pl_t\te =0;
\hspace{1cm}
 \theta(T,x) = h(x).
}
In \rf{eqPIDEFBSDE1}, $x\in\Rnu^n$,  
and the equation is $\Rnu^m$-valued. Further,
$\te$,  $\pl_x\te$, $\pl_t\te$, and $\pl^2_{xx}\te$ are everywhere evaluated at $(t,x)$ 
(we omit the arguments $(t,x)$
to simplify the equation). As before, $\pl_x\te$ is understood
as a matrix whose $(ij)$th component is $\pl_{x_j}\te^i$, and the first term in \rf{eqPIDEFBSDE1}
is understood as the  multiplication of the matrix $\pl_x\te$ by the vector-valued function following after it. Furthermore,
$\tr(\pl^2_{xx}\te\,\sg(t,x,\te)\sg(t,x,\te)^\top)$ is the vector whose $i$th component is
 the trace of the matrix  $\pl^2_{xx}\te^i \sg\sg^\top$. Finally, for any $v\in \C_b([0,T]\x \Rnu^n)$, 
 we define the function
 \aaa{
 \lb{nl}
 \tet_v(t,x) = v(t,x+\ffi(t,x,v(t,x),\fdot)) - v(t,x).
 }
 By introducing the time-changed function $u(t,x) = \te(T-t,x)$, we transform problem \rf{eqPIDEFBSDE1} to the following
Cauchy problem:
\eq{
\lb{Cauchy}
 \pl_xu \big\{ \int_{\Rnu^l_*} {\hat \ffi} (t,x,u,y)\nu(dy) - \hat f(t,x,u,\pl_x  u\, \hat \sg(t,x, u), \tet_u (t,x))\big\}
\\
   - \frac12 \tr\big(\pl^2_{xx}u\,\hat\sg(t,x,u) \hat\sg(t,x,{u})^\top\big)
-  \hat g(t,x,u,\pl_xu\, \hat\sg(t,x, u), \tet_u (t,x))  \\
-  \int_{\Rnu^l_*}  \tet_u(t,x)(y) \nu(dy)    + \pl_t u= 0; 
\hspace{1cm}
 u(0,x) = h(x).
}
In \rf{Cauchy}, $\hat f(t,x,u,p,w) = f(T-t,x,u,p,w)$, and the functions $\hat\sg$, $\hat\ffi$, and $\hat g$
are defined via $\sg$, $\ffi$, and, respectively, $g$ in the similar manner. Furthermore, the function
$\tet_u$ is defined by \rf{nl} via the function $\hat \ffi$ (however, we use the same character $\tet$).

Let us observe that problem \rf{Cauchy} is, in fact, non-local Cauchy problem \rf{PDE}-\rf{cauchysystem}
if we define the coefficients $a_{ij}$, $a_i$, $a$, and the function $\tet_u$ by formulas \rf{coeff},  and assume that
the normed space $E$ is $L_2(\nu, {\Rnu^l_*} \to \Rnu^m)$. In other words, formulas \rf{coeff} embed the PIDE
in \rf{Cauchy} into the class of non-local PDEs considered in the previous section. 
Further, it will be shown that assumptions (B1)--(B8) below imply (A1')--(A12').

As before, $\mu(s)$, $\hat \mu(s)$, $\td \mu(s)$, $P(s,r,t)$, $\sig(r)$, and $\eps(s,r)$
are continuous functions,  defined for positive arguments,
taking positive values, and non-decreasing (except $\hat\mu(s)$) with respect to each argument, 
whenever the other arguments are fixed; the
function $\hat \mu(s)$ is non-increasing. Further, 
$\mc {\td R}$, $\mc {\td R}_1$, $\mc {\td R}_2$, and $\mc {\td R}_3$ 
 are regions defined as in the previous section with $E = L_2(\nu, {\Rnu^l_*} \to \Rnu^m)$:
\aa{
& \mc {\td R} = [0,T] \x \Rnu^n \x \Rnu^m \x \Rnu^{m \x n}\x L_2(\nu, {\Rnu^l_*} \to \Rnu^m); 
\quad  \mc {\td R}_1 = [0,T] \x \Rnu^n \x \Rnu^m;  
\\
& \mc {\td R}_2 =  [0,T] \x \Rnu^n \x \{|u|\lt C_1\} \x \{|p|\lt C_2\} \x \{\|w\|_\nu \lt C_3\}; \\
& \mc {\td R}_3 = \ovl \F_T \x \{|u|\lt C_1\} \x \{|p|\lt C_2\}\x \{\|w\|_\nu \lt C_3\},
}
where, $C_1$, $C_2$, $C_3$ are constants, and  $\|\cdot\|_\nu$ 
is the norm in $L_2(\nu, {\Rnu^l_*} \to \Rnu^m)$.

We assume:
\bi
\item[\bf (B1)] 
$\hat \mu(|u|)I \lt \sg(t,x,u)\sg(t,x,u)^\top \lt \mu(|u|)I$  for all $(t,x,u)\in \mc {\td R}_1$.
\item[\bf (B2)] $(t,x,u) \mto \ffi(t,x,u,\fdot)$ is a map $\mc {\td R}_1 \to L_2(\nu, Z \to \Rnu^n)$, 
where $Z\sub\Rnu^l$ is a common support of the $L_2$-functions $y\mto \ffi(t,x,u,y)$, 
which is assumed to be of finite $\nu$-measure. 
Further, $\pl_x\ffi$ and $\pl_u\ffi$ exist for $\nu$-almost each $y$;
$\pl_t\ffi$, $\pl^2_{ux}  \ffi$,
and $\pl^2_{uu}\ffi$ exist w.r.t. the  $L_2(\nu, Z \to \Rnu^n)$-norm.  Moreover, all the mentioned derivatives
 are bounded as maps  $\mc {\td R}_1 \to L_2(\nu, Z \to \Rnu^n)$.
\item[\bf (B3)]
There exist constants $c_1,c_2,c_3>0$
and a function $\zeta:  \mc {\td R} \x \Rnu^n \to (0,+\infty)$  such that for all $(t,x,u,p,w) \in  \mc {\td R}$,
$\zeta(t,x,u,p,w,0) = 0$  and
\aa{
\big(g(t,x,u,p,w), \, u\big) \lt c_1 + c_2 |u|^2 + c_3\|w\|_\nu^2 + \zeta(t,x,u,p,w,p^{_\top} u).
}
\item[\bf (B4)] The final condition $h:\Rnu^n \to \Rnu^m$ is of class $\C^{2+\beta}_b(\Rnu^n)$, $\beta\in (0,1)$.
\item[\bf (B5)] For all  $(t,x,u,p,w)\in \mc {\td R}$,
\aa{    
&\Big|\int_Z  \ffi (t,x,u,y)   \nu(dy)\Big| \lt \sig(|u|);  \quad 
  |f(t,x,u,p,w)|  \lt \eta(|u|, \|w\|_\nu)(1+|p|); \\
&|g(t,x,u,p,w)|  \lt \big(\eps(|u|, \|w\|_\nu) + P(|u|, |p|, \|w\|_\nu)\big)(1+|p|)^2,
}where $\lim_{r\to \infty} P(s,r,q) = 0$ and $4(1+s)(1+\mu(s)) \eps(s,r) < \hat\mu(s)$.
 \item[\bf (B6)] There exist continuous derivatives $\pl_x \sg$ and $\pl_u \sg$ such that
 \aa{
  \max \big\{\big| \pl_x \sg(t,x,u)\big| , \big|\pl_u \sg(t,x,u)\big| \big\}\lt \td \mu(|u|).
  }
 \item[\bf (B7)]  For any bounded domain $\F\sub\Rnu^n$ and for any $u\in \C^{0,1}_0(\ovl\F_T)$, 
 it holds that  (1) $D(t,x,y)>0$ for $(\la_{n+1} \ox \nu)$-almost all $(t,x,y) \in \ovl\F_T \x Z$  and (2)
 $\int_Z D^{-1} (t,x,y)\nu(dy) < \La$, where $\La$ is a constant depending on $u$ and $\F$ and   
$D(t,x,y) = |\det\{I + \pl_x \ffi(t,x,u(t,x),y) + \pl_u \ffi(t,x,u(t,x),y) \pl_x u(t,x)\}|$.
 \item[\bf (B8)]  
 The functions (a)   $\pl_t f$, $\pl_t g$,  $[g]^x_\beta$,  $\pl_t \sg$,  $\pl^2_{xt} \sg$, $\pl^2_{ut} \sg$, and 
(b) $\pl^2_{xx} \sg$, $\pl^2_{xu} \sg$, $\pl^2_{uu} \sg$, 
 $\pl_x f$,  $\pl_u f$,  $\pl_p f$, $\pl_w f$, $\pl_u g$, $\pl_p g$, $\pl_w g$,
$\pl^2_{p x} f$, $\pl^2_{p u} f$,  $\pl^2_{p p} f$,  $\pl^2_{p w} f$, $\pl^2_{p x} g$,  $\pl^2_{p u} g$, 
$\pl^2_{pp} g$, $\pl^2_{pw} g$ exist and are bounded and continuous 
 in regions of form $\mc {\td R}_2$; the derivatives of group (b) are
 $\al$-H\"older continuous in $x$, $u$, $p$, $w$ for some $\al\in (0,1)$, and  all the 
 H\"older  constants are bounded over regions of form $\mc {\td R}_3$.
Further, $f$, $g$, $\pl_p f$ and $\pl_p g$ are  locally Lipschitz in $w$, and
all the Lipschitz  constants are bounded over regions of form $\mc {\td R}_2$. 
 \ei
Theorem \ref{PIDEsolution} below is the existence and uniqueness result for final value problem
\rf{eqPIDEFBSDE1} which involves a PIDE. It can be regarded as a  particular case  of Theorem \ref{existence-cauchy}
and is the main tool to show the existence and uniqueness for FBSDEs with jumps.
In particular, it is shown that assumptions (A8')--(A12'), including decompositions \rf{repa10}, \rf{tdtet},
and inequality \rf{i-w}, are fulfilled when $\tet_\te$ is given by \rf{nl}.
 \begin{thm}
 \lb{PIDEsolution}
Let (B1)--(B8) hold.  Then, final value problem  \rf{eqPIDEFBSDE1}
has a unique $\C^{1,2}_b([0,T]\x \Rnu^n)$-solution. 
\end{thm}
\begin{proof}
Since problem \rf{eqPIDEFBSDE1} is equivalent to problem \rf{Cauchy}, it suffices to prove the existence and uniqueness for the latter.
As we already mentioned, introducing functions \rf{coeff}, letting the normed space $E$ 
be $L_2(\nu, {\Rnu^l_*} \to \Rnu^m)$,
and defining $\tet_u$ by \rf{nl}, we rewrite Cauchy problem \rf{Cauchy} in form  \rf{PDE}-\rf{cauchysystem}.

Let us prove that (A1')--(A12') are implied by (B1)--(B8). Indeed, (B1) implies (A1').
Next, we note that by (B2), the function $\ffi(t,x,u,\fdot)$ is supported in $Z$
and $\nu(Z)<\infty$.
This implies (A2') since
for any $\la\gt 0$ and for any $u\in \C_b([0,T]\x\Rnu^n)$,
\aa{
\|e^{-\la t}\tet_u(t,x)\|_\nu \lt  2\, \nu(Z) \, \sup_{[0,T]\x\Rnu^n} |e^{-\la t}u(t,x)|.
}
Further, (A3') follows from (B3) and \rf{coeff}, since for any $u\in\Rnu^m$,
 $\int_Z (w(y), u) \nu(dy) \lt \frac12 \|w\|_\nu^2 + \frac{\nu(Z)}2 |u|^2$.
Next, by (B5) and (B1),
\mm{
\big|\hat f(t,x,u, p\, \hat\sg(t,x, u), w)\big|
\lt \eta(|u|, \|w\|_\nu) \big(1+  |p|\, |\hat\sg(t,x, u)|\big) \\
\lt 
\eta(|u|, \|w\|_\nu)\, \big(1+\sqrt{\mu(|u|)}\big) (1+|p|)
}
which, together with the inequality for $\ffi$ in (B5), implies the first inequality in (A5').
The second inequality in (A5') follows, again, from (B5) and (B1) by virtue of the following estimates
\aa{
\big|\hat g (t,x,u, p\, \hat\sg(t,x, u), w)\big| \lt  &\, \big(\eps(|u|, \|w\|_{\nu}) + 
P\big(|u|, \|w\|_\nu, |p|\sqrt{\mu(|u|)}\big)\big)\big(1+|p|\sqrt{\mu(|u|)}\big)^2 \\
 \lt & \, \big(\td \eps(|u|, \|w\|_{\nu}) + \td P(|u|, \|w\|_\nu, |p|)\big) (1+|p|)^2, \\
\text{and} \quad
\Big|\int_Z w(y) \nu(dy)\Big| \lt &\, \hat P(\|w\|_\nu, |p|) (1+|p|)^2, 
}
where $\td\eps(s,r) = 2\eps(s,r)(1+\mu(s))$, $\td P(s,r,q)= 2P(s,r,p\sqrt{\mu(s)})(1+\mu(s))$, and 
$\hat P(s,r) =  \nu(Z)^\frac12 \, s \, (1+r)^{-2}$.
Further, (B6) and (B8) imply (A6'), (A7'), and (A12').
 Remark, that (A7') is implied, in particular, by the fact that the function
 $L_2(\nu, Z\to \Rnu^m) \to \Rnu^m$, $w\mto \int_Z w(y) \nu(dy)$ 
 is G\^ateaux-differentiable and Lipschitz.

 It remains to verify assumptions (A8')--(A11'). 
 Let us start with (A8'). First remark that if $u\in \C^{1,2}_0(\ovl \F_T)$, where $\F$ is
 a bounded domain, then it can be extended by $0$ outside of $\F$ defining a bounded
 continuous function on $[0,T]\x\Rnu^n$. Therefore, $\tet_u(t,x)$ is well-defined on 
 $[0,T]\x\Rnu^n$ for any function $u$ which is zero on $\pl \F$.
 Further, note that by (B2), 
 $\tet_u(t,x)$ takes values in  $L_2(\nu,Z\to \Rnu^m)$ for any $u\in \C^{1,2}_b([0,T]\x\Rnu^n)$. 
Furthermore, (B2) implies that $\pl_t \tet_u(t,x)$ and $\pl_x \tet_u(t,x)$ exist in $L_2(\nu,Z\to \Rnu^m)$
and are expressed via
 $\pl_t u$, $\pl_x u$, $\pl_t \ffi$, $\pl_x\ffi$, and $\pl_u \ffi$. Hence, (A8') is fulfilled.

Let us verify (A9').  Recall that (A9') is assumption (A10) from subsection \ref{ut-est} valid for any bounded domain $\F$.
Let $u\in \C^{0,1}_0(\ovl\F_T)$ and $v(t,x)=  (\Dl t)^{-1} (u(t',x) - u(t,x))$ with 
$t' = t+\Dl t$.
The immediate computation implies decomposition \rf{repa10} with
\eqq{
\hat \tet_v = v(t, x+\hat \ffi(t,x,u(t,x),\fdot)) - v(t,x), \\
 \zeta_{u,u_x}
= \int_0^1 d\la\, \pl_x u(t', x+ \la \Dl\hat\ffi)  \int_0^1 d\bar\la \,\pl_u\hat\ffi(t,x, \bar\la u(t',x) + (1-\bar\la) u(t,x),\fdot),\\
\xi_{u,u_x} =  \int_0^1 d\la\, \pl_x u(t', x+ \la \Dl\hat\ffi) \int_0^1 d\bar\la \,
\pl_t\hat\ffi(t+\bar\la\Dl t,x, u(t',x), \fdot),
} 
where $\Dl\hat\ffi = \hat\ffi(t',x,u(t',x),\fdot) - \hat\ffi(t,x,u(t,x),\fdot)$.
Further, inequality \rf{i-w} follows from (B8). Indeed, 
 define the functions $\Phi_{t,y}(x) =x+\hat \ffi(t,x,u(t,x),y)$  
 and $\td v(t,x,y) = v(t,\Phi_{t,y}(x))$ on $[0,T]\x\Rnu^n \x Z$.
By the definition (see (B8)), $D(t,x,y) = |\det \pl_x \Phi_{t,y}(x)|$.
We have
\mmm{
 \lb{above1}
\int_{\F^\al_\tau} \Big( \int_Z |\td v|^2(t,x,y) \nu(dy)\Big)^2  dt \, dx\lt  
\int_{\F^\al_\tau} \Big( \int_{\{y: |\td v|\lt |v|\}} |\td v|^2(t,x,y) \nu(dy)\Big)^2  dt \, dx   \\
  + \int_{\F^\al_\tau}\Big( 
  \int_{\{y: |\td v|> |v|\}} |\td v|^2(t,x,y) \sqrt{D(t,x,y)}\sqrt{ D^{-1}(t,x,y)}\,\nu(dy) \Big)^2  dt \, dx
  \\
 \lt \nu(Z)^2 \int_{\F^\al_\tau} |v|^4 dx\, dt  
+ \La \int_Z \nu(dy) \int_0^\tau dt \int_{\{x: |\td v|^2>\al; D>0\}} |\td v|^4 D(t,x,y) dx \\
\lt (\nu(Z)^2 + \La\nu(Z)) \int_{\F^\al_\tau} |v|^4 dt \, dx,
}
where $\La$ is the constant from (B8) depending on $\F$ and $u$. 
The second integral in the third line is estimated as follows.
 First, we remark that since $D(t,x,y)>0$, by Theorem 1.2 in \cite{katriel} (p. 190), the map  $\Phi_{t,y}:\Rnu^n \to \Rnu^n$
 is invertible. Therefore, we can transform this integral by the change of variable $x_1 = \Phi_{t,y}(x)$.
  Thus, inequality \rf{above1} implies \rf{i-w}, and (A9') is verified.

Further, (A10') is verified immediately by \rf{nl}. To verify (A11'), we note that decomposition  
\rf{tdtet} holds with
\eqq{
\td\tet_v =  v(t, x+\hat \ffi(t,x,u(t,x),\fdot)) - v(t,x),\\
\sig_{u,u_x, u', u'_x}
= \int_0^1 d\la\, \pl_x u'(t, x+ \la \dl\hat\ffi)  \int_0^1 d\bar\la \,\pl_u\hat\ffi(t,x, \bar\la u(t,x) + (1-\bar\la) u'(t,x),\fdot),
}
where $v=u-u'$ and $\dl\hat\ffi = \hat\ffi(t,x,u(t,x),\fdot) - \hat\ffi(t,x,u'(t,x),\fdot)$. 
By (B2), $\sig_{u,u_x, u', u'_x}$ is bounded, continuous, and has a bounded derivative in $x$.
This verifies (A11').

Thus, we conclude, by Theorem \ref{existence-cauchy}, that there exists a unique $\C^{1,2}_b([0,T]\x\Rnu^n)$-solution 
to problem \rf{Cauchy}.
\end{proof}
\begin{rem}
\rm  As in Theorem \ref{existence-cauchy}, assumptions (B1)--(B7) imply the existence
of a $\C^{1,2}_b$-solution to problem \rf{eqPIDEFBSDE1}, and (B8) is required only
for the proof of uniqueness.
\end{rem}
\begin{rem}
\rm
We formulate Theorem \ref{PIDEsolution} just as a result sufficient for the application to FBSDEs. However,
we note that, by Theorem \ref{existence-cauchy},  the 
$\C^{1,2}_b$-solution to problem  \rf{eqPIDEFBSDE1}
also belongs to class
$\C^{1+\frac{\al}2,2+\al}_b([0,T]\x\Rnu^n)$
 for some $\al \in (0,\beta)$.
\end{rem}

Before we prove our main result (Theorem \ref{mainThm} below), 
which is the existence and uniqueness theorem for FBSDE \rf{eqFBSDE},
we state a version of It\^o's formula (Lemma \ref{lem19}) used in the proof of Theorem \ref{mainThm}. 
We give the proof of the lemma since we do not know a reference for the time-dependent case.
\begin{lem}
\lb{lem19}
Let $X_t$ be an $\Rnu^n$-valued semimartingale with c\`adl\`ag paths of the form
\aa{
X_t = x + \int_0^t F_s ds + \int_0^t G_s dB_s + \int_0^t \int_Z \Phi_s(y)\td N(ds \,dy),
}
where the $d$-dimensional Brownian motion $B_t$ and the compensated Poisson random measure  $\td N$ are defined as above.
Further, let $Z\sub\Rnu^l_*$ be such that $\nu(Z)<\infty$, and 
$F_t$, $G_t$, and $\Phi_t$ be stochastic processes 
with values in $\Rnu^n$, $\Rnu^{n\x d}$, and $L_2(\nu, Z\to\Rnu^n)$, respectively.
Then, for a real-valued function $\phi(t,x)$ of class
$\C^{1,2}_b([0,T]\x \Rnu^n)$, a.s., it holds that
\aaa{
\phi(t,X_t) = \phi(0,x) + \int_0^t \pl_s \phi(s,X_s) ds + \int_0^t (\pl_x \phi(s,X_s),F_s)ds + \int_0^t (\pl_x \phi(s,X_s),G_sdB_s)\notag \\
+ \frac12\int_0^t \tr \big(\pl^2_{xx}\phi(s,X_s) \, G_s G_s^\top\big) ds
+ \int_0^t \int_Z \big[\phi\big(s,X_{s-}+ \Phi_s(y)\big) - \phi(s,X_{s-})\big] \td N(ds\, dy)\notag \\
+ \int_0^t \int_Z \big[\phi\big(s,X_{s-}+ \Phi_s(y)\big) - \phi(s,X_{s-}) - (\pl_x \phi(s,X_{s-}), \Phi_s(y))\big]\nu(dy)\, ds. 
\lb{ito}
}
\end{lem}
\begin{rem}
\rm
In the above lemma, we agree that $X_{0-} = X_0 = x$.
\end{rem}
\begin{proof}[Proof of Lemma \ref{lem19}]
Let us first assume that the function $\phi$ does not depend on $t$. Applying It\^o's formula (see Theorem 33 in \cite{protter}, p. 74),
we obtain
\mmm{
\lb{ito1}
\phi(X_t) - \phi(x) =   \int_0^t (\pl_x \phi(X_s),F_s)ds + \int_0^t  (\pl_x \phi(X_{s-}),dX_s) \\
+ \frac12\int_0^t \tr \big(\pl^2_{xx}\phi(X_s) \, G_s G_s^\top\big) ds
+ \sum_{0<s\lt t} \big(\phi(X_s) - \phi(X_{s-}) - (\pl_x\phi(X_{s-}),X_s - X_{s-})\big).
}
Note that the last summand in \rf{ito1} equals 
$\int_0^t \int_Z \big[\phi(X_{s-} + \Phi(s,y)) - \phi(X_{s-}) - (\pl_x\phi(X_{s-}),\Phi_s(y))\big] N(ds\, dy)$.
By the standard argument (see, e.g., \cite{appelbaum}, p. 256),
we obtain formula \rf{ito} without the term containing $\pl_s \phi(s,X_s)$.

Now take a partition of the interval $[0,t]$. Then, for each pair of successive points,
\mmm{
\lb{successive}
\phi(t_{n+1},X_{t_{n+1}}) - \phi(t_n,X_{t_n}) = \big[\phi(t_{n+1},X_{t_n}) - \phi(t_n,X_{t_n})\big] \\
+ \big[\phi(t_{n+1},X_{t_{n+1}}) -  \phi(t_{n+1},X_{t_n})\big].
} 
The first difference on the right-hand side equals $\int_{t_n}^{t_{n+1}} \pl_s \phi(s, X_{t_n}) ds$, while 
the second difference is computed by formula \rf{ito1}. Assume, the mesh of the partition goes to zero as $n\to \infty$.
Then, summing identities \rf{successive} and letting $n\to\infty$,
we arrive at formula \rf{ito}.
 Indeed, the convergence of the stochastic integrals holds in $L_2(\Om)$ 
by Lebesgue's dominated convergence theorem, implying the convergence almost surely for a subsequence.
Further, in the term containing the time derivative $\pl_s\phi$,
we have to take into account that  $X_t$ has c\`adl\`ag paths.
\end{proof}
Let $\mc S$ denote the class of processes $(x_t, y_t, z_t,\td z_t)$ with values in $\Rnu^n$, $\Rnu^m$, $\Rnu^{m\x n}$,
and $L_2(\nu,{\Rnu^l_*}\to \Rnu^m)$, respectively, such that  $x_t, y_t$, and $z_t$ are $\mc F_t$-adapted,
$\td z_t$ is $\mc F_t$-predictable, and
\aaa{
\lb{as3}
\sup_{t\in[0,T]}\big\{ \E|x_t|^2 + \E|y_t|^2\big\} + \int_0^T\hspace{-1mm} \big( \E |z_t|^2 + 
\E\|\td z_t\|^2_\nu\big) dt < \infty.
}
The main result of this work is the following.
 \begin{thm}
 \lb{mainThm}
  Assume (B1)--(B8). Then, there exists  a solution $(X_t, Y_t, Z_t, \td Z_t)$ to
  FBSDE \rf{eqFBSDE}, such that $X_t$ is a c\`adl\`ag solution to 
 \mmm{
X_t = x + \int_0^t f\big(s,X_s,\te(s,X_s),\pl_x\te(s,X_s) \sg(s,X_s,\te(s,X_s)), \tet_\te (s,X_s) \big)ds\\ 
+\int_0^t \sg (s,X_s,\te(s,X_s)) dB_s  + \int_0^t \int_{\Rnu^l_*} \ffi(s,X_{s-},\te(s,X_{s-}),y) \td N(ds\, dy),
\label{FSDE}
} 
where $\te(t,x)$ is the unique $\C^{1,2}_b([0,T],\Rnu^n)$-solution to problem
\rf{eqPIDEFBSDE1}, whose existence was established by Theorem \ref{PIDEsolution},
and $\tet_\te$ is given by \rf{nl}.
Furthermore, $Y_t$, $Z_t$, and $\td Z_t$ are explicitly expressed via $\te$ by the formulas
\aaa{
\lb{link}
Y_t= \te(t,X_t), \quad  Z_t = \pl_x\te(t,X_{t})\sg(t,X_{t},\te(t,X_{t})), \; \text{and} \quad
\td Z_t = \tet_\te(t,X_{t-}).
}
Moreover, the solution $(X_t,Y_t,Z_t,\td Z_t)$ is unique in the class $\mc S$.
 \end{thm}
\begin{proof}
\textit{Existence.}
First we prove that SDE \rf{FSDE} has a unique c\`adl\`ag solution.
Define $\td f(t,x) = f(t,x,\te(t,x),\pl_x\te(t,x) \sg(t,x,\te(t,x)), \tet_\te (t,x) )$,
$\td \sg(t,x) = \sg(t,x,\te(t,x))$, and $\td \ffi(t,x,y) = \ffi(t,x,\te(t,x),y)$. With this notation, SDE \rf{FSDE} becomes
\aaa{
\lb{SDE-mod}
X_t = x+ \int_0^t \td f(t,X_s) ds + \int_0^t \td \sg(s,X_s) dB_s + \int_0^t \int_{\Rnu^l_*} \td\ffi(s,X_{s-},y) \td N(ds\, dy).
}
Note that since $\te$ is of class $\C^{1,2}_b([0,T]\x\Rnu^n)$,  (B1) and (B5) imply that $\td f(t,x)$, $\td \sg(t,x)$,
$\int_Z \td \ffi(t,x,y) \nu(dy)$ are bounded. 
Further, (B6) implies the boundedness of $\pl_x \td \sg(t,x)$, while (B1), (B2), (B6), and (B8) imply 
the boundedness of $\pl_x \td f(t,x)$.
Furthermore, (B2) implies the boundedness of $\pl_x \int_{\Rnu^l_*} \td \ffi(t,x,y)\nu(dy)$. 
Therefore, the Lipschitz condition and the linear growth conditions,
required for the existence and uniqueness of a c\`adl\`ag adapted solution to \rf{SDE-mod} 
(see \cite{appelbaum}, Theorem 2.6.9, p. 374), are fulfilled.
By Theorem 2.6.9 in \cite{appelbaum} (more precisely, by its time-dependent version considered in Exercise 2.6.10, p. 375),
there exists a unique c\`adl\`ag solution $X_t$ to SDE \rf{SDE-mod}.

Further, define $Y_t$, $Z_t$, and $\td Z_t$ by formulas \rf{link}. Applying  
 It\^o's formula (Lemma \ref{lem19}) to $\te(t,X_t)$, we obtain
\mmm{
\label{eqBSDE} 
 \te(t, X_t)= \te(T,X_T)  - \int_t^T \pl_x\te (s,X_{s-}) \sg(s, X_{s-},\te(s,X_{s-})) \, dB_s    \\
- \int_t^T \Big\{ \pl_x\te(s,X_s) f\big(s,X_s,\te(s,X_s),\pl_x\te(s,X_s) \sg(s,X_s,\te(s,X_s)), \tet_\te (s,X_s) \big)   \\
+ \pl_x\te(s,X_s)\int_{\Rnu^l_*}  \ffi(s,X_s,\te(s,X_s), y)\nu(dy) + \pl_s \te (s, X_s)  \\ 
+ \frac12 \tr\big[\pl^2_{xx}\te(s,X_s)  \sg(s,X_s,\te(s, X_s))   \sg(s,X_s,\te(s, X_s))^\top\big]
+ \int_{\Rnu^l_*} \tet_\te(s,X_s)(y)\nu(dy) \Big\} \, ds \\
- \int_t^T  \int_{\Rnu^l_*}\tet_\te(s,X_{s-})(y) \td N(ds, dy). 
}
Since $\te(t,x)$ is a solution to PIDE \rf{eqPIDEFBSDE1}, then $Y_t$, $Z_t$, and $\td Z_t$, defined by \rf{link}, 
are solution processes for the BSDE
in \rf{eqFBSDE}. Furthermore, $Y_t$ and $Z_t$ are c\`adl\`ag, and $\td Z_t$ is left-continuous with right limits.

\textit{Uniqueness.}
Suppose $(X'_t,Y'_t, Z'_t, \td Z'_t)$ is another solution to FBSDE \rf{eqFBSDE} satisfying \rf{as3}.
 As before, $\te(t,x)$ is the unique  $\C^{1,2}_b([0,T],\Rnu^n)$-solution to PIDE \rf{eqPIDEFBSDE1}.
 Define $(Y''_t, Z''_t,\td Z''_t)$ by formulas \rf{link}
via $\te(t,x)$ and $X'_t$. Therefore, $(Y'_t, Z'_t, \td Z'_t)$ and  $(Y''_t, Z''_t,\td Z''_t)$ are two solutions to the BSDE in 
\rf{eqFBSDE} with the process $X'_t$ being fixed. By the results of \cite{tang} (Lemma 2.4, p.1455), the solution to the BSDE
in \rf{eqFBSDE} is unique in the class of processes $(Y_t,Z_t,\td Z_t)$ whose squared norm 
$\sup_{t\in[0,T]}\E|Y_t|^2 + \int_0^T\hspace{-1mm} \big( \E |Z_t|^2 + \E\|\td Z_t\|^2_\nu\big) dt$ is finite.
Without loss of generality,
we can assume that $Y'_t$ is  c\`adl\`ag by considering, if necessary, its  c\`adl\`ag modification.
Since both $Y'_t$ and $Y''_t$ are  c\`adl\`ag, then $Y'_t = \te(t,X'_t)$ for all $t\in [0,T]$ a.s.
Further, $Z'_t = Z''_t$  and $\td Z'_t = \td Z''_t$ as elements of $L_2([0,T] \x \Om \to \Rnu^m)$
and $L_2([0,T] \x \Om \to E)$, respectively, where $E= L_2(\nu, \Rnu^l_* \to\Rnu^m)$. 
Therefore, 
$\la \ox \PP$-a.e. ($\la$ is the Lebesgue measure on $\Rnu$), $f(t,X'_t, Y'_t,Z'_t, \td Z'_t) = f(t,X'_t, Y''_t,Z''_t, \td Z''_t)$
 and
consequently, $\int_0^t f(s,X'_s, Y'_s, Z'_s, \td Z'_s) ds =  \int_0^t f(s,X'_s, Y''_s, Z''_s, \td Z''_s) ds$ for each $t\in [0,T]$ a.s.
Hence, $X'_t$ is a solution to SDE \rf{FSDE}. By uniqueness, $X'_t = X_t$.
Thus, the quadruple $(X_t,Y_t,Z_t,\td Z_t)$ is unique in the class $S$.
\end{proof}
\begin{rem}
\rm 
Finally, we would like to remark that the monotonicity assumption in \cite{wu99} (Assumption H3.2, p. 436) is not fulfilled
for FBSDE \rf{eqFBSDE} if $m\ne n$. 
Indeed, rewrite H3.2 in our notation and separate the terms on its left-hand side. We obtain
\mmm{
\lb{H32}
-(\Dl g, \Dl x) + (\Dl f,\Dl y) + (\Dl \sg, \Dl z) + \int_{\Rnu^l_*}(\Dl \ffi(\xi), \Dl \td z(\xi)) \nu(d\xi) \\
\lt -\beta_1|\Dl x|^2 - \beta_2(|\Dl y|^2 + |\Dl z|^2 + \|\Dl \td z\|^2_\nu),
}
where $\beta_1, \beta_2\gt 0$ are constants,
$\Dl x = G(x-x')$, $\Dl y = G^\top(y-y')$, $\Dl z = G^\top(z-z')$, $\Dl \td z = G^\top(\td z-\td z')$,
$\Dl \sg = \sg(t,x,y) - \sg(t,x',y')$, $\Dl \ffi = \ffi(t,x,y,\fdot) - \ffi(t,x',y',\fdot)$, 
$\Dl g = g(t,x,y,z,\td z) - g(t,x',y',z',\td z')$, $\Dl f = f(t,x,y,z,\td z) - f(t,x',y',z',\td z')$, 
and $G$ is an $m\x n$ matrix of full rank. The above inequality is assumed to be
fulfilled for all $(t,x,y,z,\td z)$, $(t,x',y',z',\td z') \in [0,T]\x \Rnu^n \x\Rnu^m\x\Rnu^{m\x n}\x 
L_2(\nu,{\Rnu^l_*}\to \Rnu^m)$.

 Let $m<n$. Then, according to \cite{wu99}, $\beta_2>0$.
Consider \rf{H32} in the case $x=x'$, $y=y'$, $z\ne z'$. Then, the left-hand side of \rf{H32} equals  zero while the right-hand side is negative. 

If $m > n$, we have to additionally assume that $\sg(t,x,y)$ explicitly depends on $y$.
Consider the case $x=x'$, $\td z = \td z'$, $y-y' \in \ker G^\top$, $z\ne z'.$
Then, \rf{H32} implies
\aa{
(G (\sg(t,x,y) - \sg(t,x,y')), z-z') \lt 0.
}
Since $\ker G = \{0\}$, this inequality cannot be fulfilled for all $y,y',z,z'$ as above.

\end{rem}

\section*{Acknowledgments}
The second-named author was supported by the PhD scholarship no. SFRH/BD/51172/2010 from the FCT 
(Funda\c c\~ao para a Ci\^encia e Tecnologia).


\bibliographystyle{plain}

\end{document}